\documentclass[11pt,reqno]{amsart}

\usepackage{amsmath,amsthm,amssymb,comment,fullpage}
\usepackage{braket}
\usepackage{mathtools}

\usepackage{caption}
\usepackage{times}
\usepackage[T1]{fontenc}
\usepackage{mathrsfs}
\usepackage{latexsym}
\usepackage[dvips]{graphics}
\usepackage{epsfig}
\usepackage{amsmath,amsfonts,amsthm,amssymb,amscd}
\usepackage{bm}
\input amssym.def
\input amssym.tex
\usepackage{color}
\usepackage{hyperref}
\usepackage{url}
\newcommand{\bburl}[1]{\textcolor{blue}{\url{#1}}}

\usepackage{tikz}
\usepackage{tkz-tab}
\usepackage{tkz-graph}
\usetikzlibrary{shapes.geometric,positioning}

\newcommand{\burl}[1]{\textcolor{blue}{\url{#1}}}

\newcommand\norm[1]{\left\lVert#1\right\rVert}

\numberwithin{equation}{section}

\newtheorem{thm}{Theorem}[section]
\newtheorem{conj}[thm]{Conjecture}
\newtheorem{cor}[thm]{Corollary}

\newtheorem{exa}[thm]{Example}
\newtheorem{defi}[thm]{Definition}

\theoremstyle{plain}

\newtheorem{example}[thm]{Example}
\newtheorem{lemma}[thm]{Lemma}

\newtheorem{rem}[thm]{Remark}

\theoremstyle{definition}

\newcommand\be{\begin{equation}}
\newcommand\ee{\end{equation}}
\newcommand\bee{\begin{equation*}}
\newcommand\eee{\end{equation*}}
\newcommand\bea{\begin{eqnarray}}
\newcommand\eea{\end{eqnarray}}
\newcommand\beae{\begin{eqnarray*}}
\newcommand\eeae{\end{eqnarray*}}
\newcommand\bi{\begin{itemize}}
\newcommand\ei{\end{itemize}}
\newcommand\ben{\begin{enumerate}}
\newcommand\een{\end{enumerate}}
\newcommand\bc{\begin{center}}
\newcommand\ec{\end{center}}
\newcommand\ba{\begin{array}}
\newcommand\ea{\end{array}}

\newcommand{\R}{\ensuremath{\mathbb{R}}}

\newcommand{\N}{\mathbb{N}}

\newcommand\frakfamily{\usefont{U}{yfrak}{m}{n}}
\DeclareTextFontCommand{\textfrak}{\frakfamily}

\newcommand{\hr}[1]{\href{#1}{\url{#1}}}

\usepackage{thm-restate}

\makeatletter
\g@addto@macro\bfseries{\boldmath}
\makeatother
\newcommand{\tss}[1]{\textsuperscript{#1}}

\title{Crescent Configurations In Normed Spaces}

\author{Sara Fish}
\email{\textcolor{blue}{\href{mailto:sfish@caltech.edu}{sfish@caltech.edu}}}
\address{Department of Mathematics, California Institute of Technology, Pasadena, CA 91126}

\author{Dylan King}
\email{\textcolor{blue}{\href{mailto:kingda16@wfu.edu}{kingda16@wfu.edu}}}
\address{Department of Mathematics, Wake Forest University, Winston-Salem, NC 27109}

\author{Steven J. Miller}
\email{\textcolor{blue}{\href{mailto:sjm1@williams.edu}{sjm1@williams.edu}},  \textcolor{blue}{\href{Steven.Miller.MC.96@aya.yale.edu}{Steven.Miller.MC.96@aya.yale.edu}}}
\address{Department of Mathematics and Statistics, Williams College, Williamstown, MA 01267}

\author{Eyvindur A. Palsson}
\email{\textcolor{blue}{\href{mailto:palsson@vt.edu}{palsson@vt.edu}}}
\address{Department of Mathematics, Virginia Tech, Blacksburg, VA 24061}

\author{Catherine Wahlenmayer}
\email{\textcolor{blue}{\href{mailto:wahlenma001@knights.gannon.edu}{wahlenma001@knights.gannon.edu}}}
\address{Department of Mathematics, Gannon University, Erie, PA 16541}

\thanks{This work was partially supported by NSF grants DMS1659037 and DMS1561945, Simons Foundation Grant \#360560, the N.S Reynolds Scholarship Committee at Wake Forest University, and Williams College. We thank Charles Devlin for helpful conversations about this problem.}



\date{\today}

\begin{document}

\maketitle

\begin{abstract} 

We study the problem of crescent configurations, posed by Erd\H{o}s in 1989. A crescent configuration is a set of $n$ points in the plane such that: 1) no three points lie on a common line, 2) no four points lie on a common circle, 3) for each $1 \leq i \leq n - 1$, there exists a distance which occurs exactly $i$ times. Constructions of sizes $n \leq 8$ have been provided by Liu, Pal\' asti, and Pomerance. Erd\H{o}s conjectured that there exists some $N$ for which there do not exist crescent configurations of size $n$ for all $n \geq N$. 

We extend the problem of crescent configurations to general normed spaces $(\R^2, ||\cdot||)$ by studying \textit{strong crescent configurations in $||\cdot||$}. In an arbitrary norm $||\cdot||$, we construct a strong crescent configuration of size 4. We also construct larger strong crescent configurations in the Euclidean, taxicab, and Chebyshev norms, of sizes $n \leq 6$, $n \leq 8$, and $n \leq 8$ respectively. When defining strong crescent configurations, we introduce the notion of \textit{line-like configurations in $||\cdot||$}. A line-like configuration in $||\cdot||$ is a set of points whose distance graph is isomorphic to the distance graph of equally spaced points on a line. In a broad class of norms, we construct line-like configurations of arbitrary size. 

Our main result is a crescent-type result about line-like configurations in the Chebyshev norm. A \textit{line-like crescent configuration} is a line-like configuration for which no three points lie on a common line and no four points lie on a common $||\cdot||$ circle. We prove that for $n \geq 7$, every line-like crescent configuration of size $n$ in the Chebyshev norm must have a rigid structure. Specifically, it must be a \textit{perpendicular perturbation} of equally spaced points on a horizontal or vertical line.
\end{abstract}

\tableofcontents


\section{Introduction}\label{section1}

\subsection{Background}

The Erd\H{o}s distinct distances problem is a core problem in discrete geometry. It asks the following deceptively simple question: What is the minimum number of distinct distances determined by $n$ points in the plane? Erd\H{o}s posed this problem in a 1946 paper \cite{Erd46}, in which he proved the lower bound $\Omega(n^{1/2})$ using a simple geometrical argument and the upper   bound $O( n / \sqrt{\log{n}} )$ by considering the number of distinct distances determined by a $\sqrt{n} \times \sqrt{n}$ square lattice. Over the subsequent decades, this lower bound was gradually improved. In 2015, Guth and Katz \cite{GK} proved the lower bound $\Omega( n / \log{n} )$, solving the problem up to a factor of $\sqrt{\log{n} }$. 

The Erd\H{o}s distinct distances problem inspired many related questions. We study the problem of \textit{crescent configurations}, first posed by Erd\H{o}s in \cite{Erd89}. Consider the following question: What is the structure of a set of $n$ points which determines $n - 1$ distinct distances, such that for each $1 \leq i \leq n - 1$, the $i$\tss{th} distance occurs exactly $i$ times? For every $n$, many such sets exist. For example, consider $n$ equally spaced points on a line, or $n$ equally spaced points on a circular arc (Figure \ref{fig:l2_trivial}).\footnote{Not all instances of $n$ equally spaced points on a circle satisfy this property. For example, the set $\{ (0,1), (1,0), (-1,0), (0,-1) \}$ determines the distance $\sqrt{2}$ four times and the distance 2 two times. These exceptions are unimportant, so we ignore them.
} 

\begin{figure}[h!]
    \centering
    \includegraphics[scale=0.4]{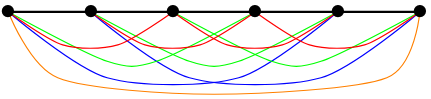}
    \includegraphics[scale=0.4]{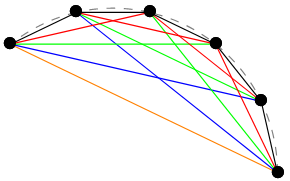}
    \caption{Equally spaced points on a line and on a circle. }
    \label{fig:l2_trivial}
\end{figure}

One might ask whether every such set must make use of the structure of lines or circles. More precisely, a set of points is said to lie in \textit{general position} if no three points lie on a common line and no four points lie on a common circle. Using this notion, we define what it means for a set of points to form a crescent configuration. 

\begin{defi}\label{def:euclcrescent}
A set of $n$ points in the plane is said to form a \textbf{crescent configuration} if the following two conditions hold.
\begin{enumerate}
    \item The $n$ points lie in general position.
    \item For each $1 \leq i \leq n - 1$, there exists a distance which occurs with multiplicity exactly $i$. 
\end{enumerate}
\end{defi}

Erd\H{o}s' question was the following: For which $n$ does there exist a crescent configuration of size $n$? Constructions of crescent configurations of size $ n \leq 8$ have been provided by Liu, Pal\' asti, and Pomerance \cite{Liu, Pal87, Pal89, Erd89}. These constructions are non-obvious and geometrically intricate. For example, Pal\' asti's crescent configuration of size 8 is depicted in Figure \ref{fig:palasti8}. 

\begin{figure}[h!]
    \centering
    \includegraphics[scale=0.17]{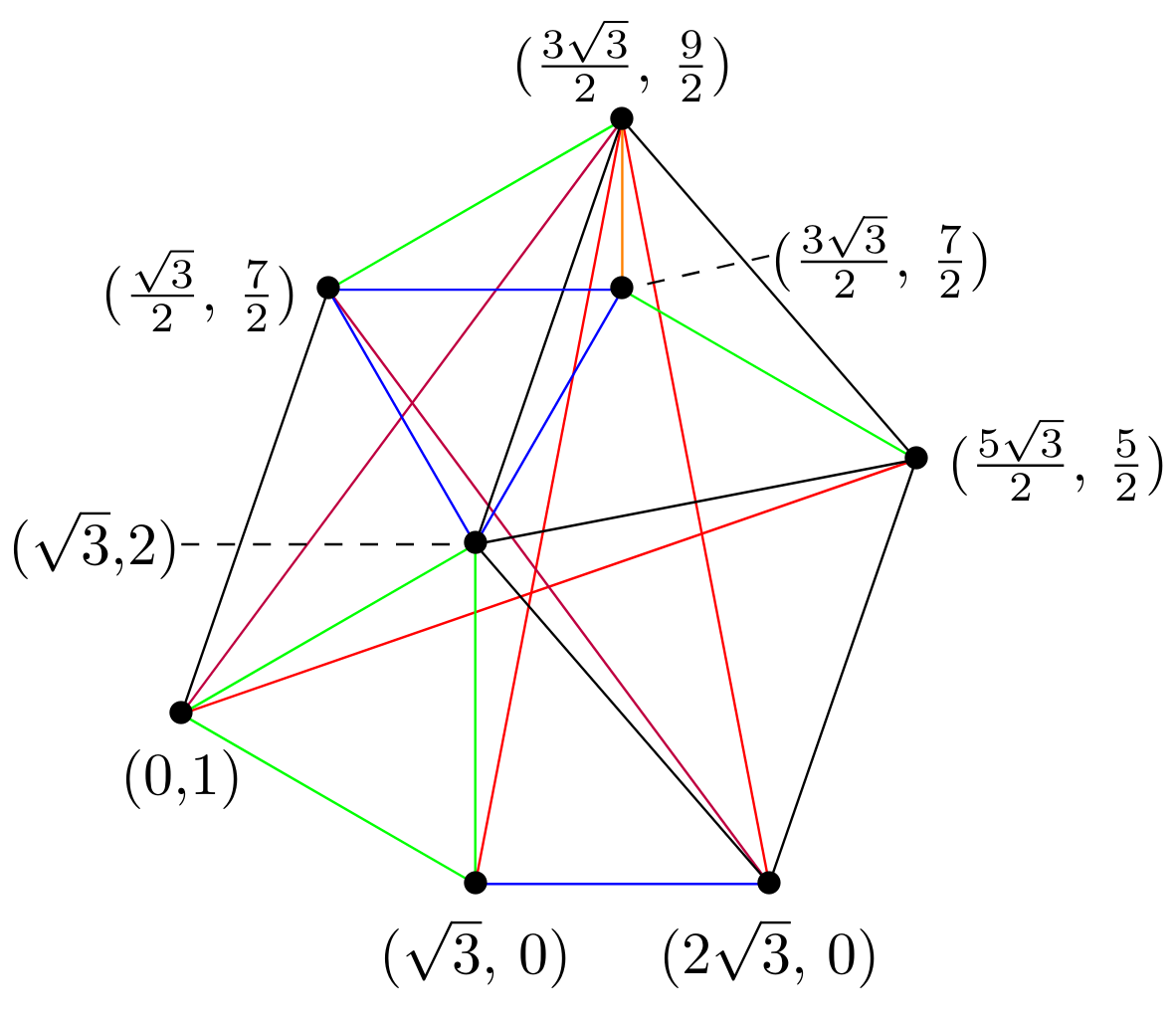}
    \caption{A crescent configuration of size 8 due to Pal\' asti \cite{Pal89}.}
    \label{fig:palasti8}
\end{figure}

The question as to whether crescent configurations of size $n$ exist remains open for $n \geq 9 $. Motivated by the observation that Pal\' asti's constructions lie on a triangular lattice, Burt et al. \cite{BGMMPS} exhaustively searched a 91 point triangular lattice and showed that it does not contain a crescent configuration of size 9. By contrast, Pal\' asti's crescent configuration of size 8 is contained in a 37 point triangular lattice.

Often, studying a distance problem in a more general normed space can reveal additional structure of the problem. A first example is the Erd\H{o}s distinct distances problem. Recall that the Erd\H{o}s distinct distances problem asks for the minimum number of distinct distances determined by $n$ points in the plane. The current best lower bound is $\Omega(n / \log n)$ by Guth and Katz in 2015 \cite{GK}, which in particular improves upon the lower bound $\Omega(n^{1/2})$ by Erd\H{o}s in 1946 \cite{Erd46} and the lower bound $\Omega(n^{4/5})$ by Sz\' ekely in 1993 \cite{Sz}. Garibaldi \cite{Ga} provided conditions for general norms in $\R^2$ to satisfy these weaker $\Omega(n^{1/2})$ and $\Omega(n^{4/5})$ bounds, leading to a deeper understanding of the techniques used in their proofs. 

A second example is the unit distances problem, first posed by Erd\H{o}s \cite{Erd46} in 1946. The original unit distances problem asks for the maximum number of distances of unit length determined by $n$ points in the plane in the Euclidean norm. It can be generalized to arbitrary norms in $\R^2$ as follows. Let $u_{||\cdot||}(n)$ denote the maximum number of distances of unit length determined by $n$ points in the plane in the norm $||\cdot||$. Brass \cite{Br} proved that if $|| \cdot ||$ is not strictly convex (i.e., the unit circle of $||\cdot||$ contains a line segment), then $u_{||\cdot||}(n) = \Theta(n^2)$. By contrast, Valtr \cite{Va} proved that if $||\cdot||$ is strictly convex, then $u_{||\cdot||}(n) = O(n^{4/3} )$. Interestingly, this upper bound $u_{||\cdot||}(n) = O(n^{4/3})$ cannot be improved without taking into account the geometry of a strictly convex norm $||\cdot||$. Valtr \cite{Va} constructed a strictly convex norm $||\cdot||$ for which $u_{||\cdot||}(n) = \Theta(n^{4/3})$. Moreover, there exist norms for which $u_{||\cdot||}(n) = o(n^{4/3})$. Matou\v{s}ek \cite{Ma} proved that ``almost every'' strictly convex norm $||\cdot||$ satisfies $u_{||\cdot||}(n) = O(n \log{n} \log{\log{n}} )$. 

Previously, crescent configurations have only been studied in the Euclidean setting. We extend the problem of crescent configurations to general normed spaces $(\R^2, ||\cdot|| )$. 


\subsection{Overview of results}

In Section \ref{section2}, we define \textit{strong crescent configurations}, a generalization of crescent configurations to normed spaces $(\R^2, ||\cdot||$). To do this, we introduce the concept of \textit{line-like configurations} and \textit{strong general position} in $||\cdot||$.

In Section \ref{section3}, we construct infinitely many line-like configurations of arbitrary size under a broad class of norms. We say that a set of $n$ points forms a \textit{line-like configuration in $||\cdot||$} if its distance graph, measured in $||\cdot||$, is isomorphic to the distance graph of $n$ equally spaced points on a line. 

\begin{restatable}{thm}{linelinelike}
\label{thm:linelinelike}
Let $||\cdot||$ be a norm which is not strictly convex. Then for each $n$, there exist infinitely many (after scaling and translating) line-like configurations of size $n$ in $||\cdot||$.
\end{restatable}

\begin{restatable}{thm}{circlelinelike}
\label{thm:circlelinelike}
Let $||\cdot||$ be a norm whose unit circle contains an arc contained in an $L^2$ circle centered at the origin. Then for each $n$, there exist infinitely many (after scaling and translating) line-like configurations of size $n$ in $||\cdot||$.
\end{restatable}

Let $||\cdot||$ be a norm which does not satisfy the conditions from Theorem \ref{thm:linelinelike} or Theorem \ref{thm:circlelinelike}. For all $n \geq 5$, we conjecture that the only line-like configurations of size $n$ in $||\cdot||$ are equally spaced points on a line (cf. Section \ref{section:uglylinelike}). 

In Section \ref{section4}, we prove a crescent-type result about crescent line-like configurations in the $L^\infty$ norm. We say a line-like configuration is a \textit{line-like crescent configuration} if no three points lie on a common line and no four points lie on a common $||\cdot||$ circle. We say that $P_1, \dots, P_n$ is a \textit{perpendicular perturbation} of a line $\ell$ if there exist equally spaced points $Q_1, \dots, Q_n$ on $\ell$ so that the lines $\overleftrightarrow{P_i Q_i} \perp \ell$ for all $1 \leq i \leq n$. 

\begin{defi}
Let $a = (a_x, a_y)$ and $b = (b_x, b_y)$ be two points in the plane. The \bm{$L^\infty$} \textbf{ norm} (Chebyshev norm) is defined by
\[ || a - b ||_{L^\infty} := \max\{ |b_x - a_x|,\ |b_y - a_y| \}. \]
\end{defi}

\begin{restatable}{thm}{crescentlinfty}
\label{thm:crescentlinfty} 
Let $n \geq 7$. Then every line-like crescent configuration in $L^\infty$ of size $n$ is a perpendicular perturbation of a horizontal or vertical line. 
\end{restatable}

For $n \leq 6$, there exist line-like crescent configurations in $L^\infty$ of size $n$ which are not perpendicular perturbations (cf. Example \ref{ex:counterexample_crescentlinfty}).

In Section \ref{section5} we provide explicit constructions of strong crescent configurations. In every norm, we construct a strong crescent configuration of size four. 

\begin{restatable}{thm}{crescentgeneral}
\label{thm:crescentgeneral}
Let $||\cdot||$ be any norm. Then there exists a strong crescent configuration of size 4 in $||\cdot||$.
\end{restatable}

We also construct larger strong crescent configurations in the $L^2$, $L^1$,and $L^\infty$ norms. The constructions were found by using a computer program to search a lattice, a technique previously employed by Pal\' asti in \cite{Pal89}. 

\begin{defi}
Let $a = (a_x, a_y)$ and $b = (b_x, b_y)$ be two points in the plane. The \bm{$L^1$} \textbf{ norm} (taxicab norm) is defined by
\[ || a - b ||_{L^1} := |b_y - a_y| + |b_x - a_x|. \]
The \bm{$L^2$}\textbf{ norm} (Euclidean norm) is defined by 
\[ || a - b ||_{L^2} := \sqrt{ (b_y - a_y)^2 + (b_x - a_x)^2 }.\]
\end{defi}

First, we provide constructions of strong crescent configurations in $L^2$. The set of strong crescent configurations in $L^2$ (from Definition \ref{def:generalcrescent}) is a subset of the set of crescent configurations (from Definition \ref{def:euclcrescent}). Crescent configurations of size $n \leq 8$ have been constructed in \cite{Liu, Pal87, Pal89, Erd89}. However, none of these constructions of sizes $n = 6,7,8$ is strong. We provide a construction of a strong crescent configuration in $L^2$ of size 6. 

\begin{restatable}{thm}{ltwocrescent}
In the $L^2$ norm, there exist strong crescent configurations of size $n \leq 6$. 
\end{restatable}

Second, we provide constructions of strong crescent configurations in $L^1$ and $L^\infty$. We do so by first using a computer program to search a square lattice for strong crescent configurations in $L^\infty$. The constructions in $L^\infty$ immediately give rise to constructions in $L^1$, as there is a dual relationship between sets of points in $L^1$ and $L^\infty$. We chose to study those norms in particular because they are highly symmetric and easily computable. Given a lattice and a method to compute distances and circles in an arbitrary norm $||\cdot||$, our algorithm would similarly be able to search for strong crescent configurations in $||\cdot||$ in a lattice. 

\begin{restatable}{thm}{linftycrescent}
In the $L^\infty$ norm, there exist strong crescent configurations of sizes $n \leq 8$.  
\end{restatable}

\begin{cor}\label{lonecrescent}
In the $L^1$ norm, there exist strong crescent configurations of sizes $n \leq 8$. 
\end{cor}


\section{General position and crescent configurations in normed spaces}\label{section2}


\subsection{Preliminaries}

Throughout this paper, we study the vector space $\R^2$, equipped with an arbitrary norm $||\cdot||: \R^2 \to \R$. In this section, we recall properties of these normed spaces $(\R^2, ||\cdot||)$ that are used in our proofs. See \cite{MSW} for a comprehensive survey of the geometry of normed spaces. 

\begin{defi}
A \textbf{norm} on $\R^2$ is a function $||\cdot||: \R^2 \to \R$ satisfying the following three properties.
\begin{enumerate}
    \item For all $x \in \R^2$ we have $||x|| \geq 0$. Moreover, $||x|| = 0$ if and only if $x = 0$.
    \item For all $x \in \R^2$ and $\lambda \geq 0$ we have $||\lambda x || = \lambda ||x ||$.
    \item For all $x,y \in \R^2$ we have $||x + y|| \leq ||x|| + ||y||$.
\end{enumerate}
Each norm $|| \cdot||: \R^2 \to \R$ specifies a \textbf{distance function} (or metric)  $d_{||\cdot||} : \R^2 \to \R$, given by 
\[ d_{||\cdot||}(x,y) := ||x - y||\]
for all $(x,y) \in \R$.
\end{defi}

A norm on $\R^2$ is uniquely determined by specifying its unit ball $B$.

\begin{defi}
A \textbf{unit ball} on $\R^2$ is a set $B \subset \R^2$ satisfying the following properties: 
\begin{enumerate}
    \item $B$ is closed and bounded,
    \item $B$ has a non-empty interior,
    \item $B$ is centrally symmetric,
    \item $B$ is convex.
\end{enumerate}
The corresponding \textbf{unit circle} is the boundary $\partial B$. We denote the \textbf{circle of radius \bm{$r$} centered at \bm{$p$}} by $B_{||\cdot||}(p, r)$. 

\end{defi}

\begin{exa}
For $1 \leq p < \infty$, the \textbf{\bm{$L^p$} norm}, denoted $||\cdot||_p$, is defined by 
\[ || (x,y) ||_p := (|x|^p + |y|^p)^{1/p}, \]
for all $(x,y) \in \R^2$. The corresponding unit ball is 
\[ B = \{ (x,y),\ |x|^p + |y|^p \leq 1 \}. \]
For $p = \infty$, the \textbf{\bm{$L^\infty$} norm}, denoted $||\cdot||_{\infty}$, is defined by 
\[ || (x,y) ||_{\infty} := \max(|x|, |y|),\]
for all $(x,y) \in \R^2$. The corresponding unit ball is 
\[ B = \{ (x,y),\ |x|,|y| \leq 1 \}. \]
Euclidean distance is given by the $L^2$ norm. 
\end{exa}

Next we briefly discuss strict convexity. An in depth treatment can be found in \cite{MSW} (pg. 10--15). 

\begin{defi}\label{def:strictlyconvex}
Let $||\cdot||$ be a norm with unit ball $B$. The following are equivalent.
\begin{enumerate}
    \item For $x,y \in \R^2$, we have $||x + y|| = ||x|| + ||y||$ if and only if $x = \lambda y$ for some $\lambda \geq 0$. 
    \item The unit circle $\partial B$ does not contain a line segment. 
\end{enumerate}
A norm which satisfies these properties is said to be \textbf{strictly convex.}
\end{defi}

\begin{exa}$ $
\begin{enumerate}
    \item $L^p$ is strictly convex for $1 < p < \infty$.
    \item $L^1$ and $L^\infty$ are not strictly convex. 
\end{enumerate}
\end{exa}

We recall the following lemma about intersection points of circles in strictly convex norms. A proof can be found in \cite{MSW} (pg. 13--14). 

\begin{lemma}\label{lemma:twocircles}
Let $||\cdot||$ be a strictly convex norm. Then two circles $B_{||\cdot||}(p_1, r_1)$ and $B_{||\cdot||}(p_2, r_2)$ with $p_1\neq p_2$ intersect in at most two points. 
\end{lemma}

\begin{lemma}\label{lemma:duality}
The spaces $(\mathbb{R}^2, ||\cdot||_\infty)$ and $(\mathbb{R}^2, ||\cdot||_1)$ are isometric.
\end{lemma}
\begin{proof}

Let $T: (\mathbb{R}^2, ||\cdot||_\infty)\to (\mathbb{R}^2, ||\cdot||_1)$ be the linear map given by the matrix $\begin{bmatrix}1 & 1\\ 1 & -1\end{bmatrix}$. For each $(x,y) \in \R^2$ we have
$$
\norm{  \begin{bmatrix} 
   1 & 1 \\
  1 & -1  \\
   \end{bmatrix} 
   \begin{bmatrix}
   x\\
   y\\
   \end{bmatrix}}_\infty = \max\{|x+y|, |x-y|\} = |x|+|y| = \norm{\begin{bmatrix} x\\y\end{bmatrix} }_1. 
$$   

This establishes an isometric map between the $L^\infty$ unit ball and the $L^1$ unit ball. 
\end{proof}

\subsection{Line-like configurations}\label{section:linelikedef}

 We first recall the notions of general position and crescent configurations in the Euclidean setting. Crescent configurations were first studied by Erd\H{o}s in \cite{Erd89}, and the term ``crescent configuration'' was coined by Burt et al. in \cite{BGMMPS}. 

\begin{defi}\label{def:euclgeneralposition}
A set of points in the plane is said to lie in \textbf{general position} if no three points lie on a common line and no four points lie on a common circle. 
\end{defi}

\noindent \textbf{Definition \ref{def:euclcrescent}}\textbf{.} \textit{A set of $n$ points in the plane is said to form a \textbf{crescent configuration} if the following two conditions hold.
\begin{enumerate}
    \item The $n$ points lie in general position.
    \item For each $1 \leq i \leq n - 1$, there exists a distance which occurs with multiplicity exactly $i$. 
\end{enumerate}}

We want to generalize the notion of crescent configurations to a general normed space. In the Euclidean setting, $n$ equally spaced points on a line and $n$ equally spaced points on a circular arc (Figure \ref{fig:l2_trivial}) satisfy crescent configuration condition (2). The purpose of condition (1) is to omit these trivial configurations. The following example demonstrates that there exist trivial constructions in other norms which satisfy Definition \ref{def:euclcrescent}. For larger classes of examples, see Sections \ref{section:linelinelike} and \ref{section:circlelinelike}. 

\begin{example}\label{ex:loobadlinelike}
Consider the $L^\infty$ norm. For each $n$, there exist infinitely many sets of $n$ points which satisfy crescent configuration condition (2), and satisfy the property that no three points lie on a line and no four points lie on an $L^\infty$ ball. To construct such a set, start with $n$ equally spaced points on a horizontal line, say 

\[ (1,0), (2,0), (3,0), \dots, (n, 0). \]

Then perturb the points in the $y$ direction. Specifically, pick $\epsilon_1, \dots, \epsilon_n \in \R$ so that the point set 

\[ (1,\epsilon_1), (2,\epsilon_2), \dots, (n, \epsilon_n) \]

satisfies the following two properties. \begin{enumerate}
    \item For all $i,j \in \{ 1, \dots, n \}$, $d_{L^\infty}(\ (i,\epsilon_i), (j, \epsilon_j)\ ) = |j-i|$. 
    \item No three points lie on a line. 
\end{enumerate}

For example, this can be accomplished by picking $\epsilon_i = 1/i$. See Figure \ref{fig:loobadlinelike}. 

\begin{figure}[h!]
    \centering
    \includegraphics[scale=0.3]{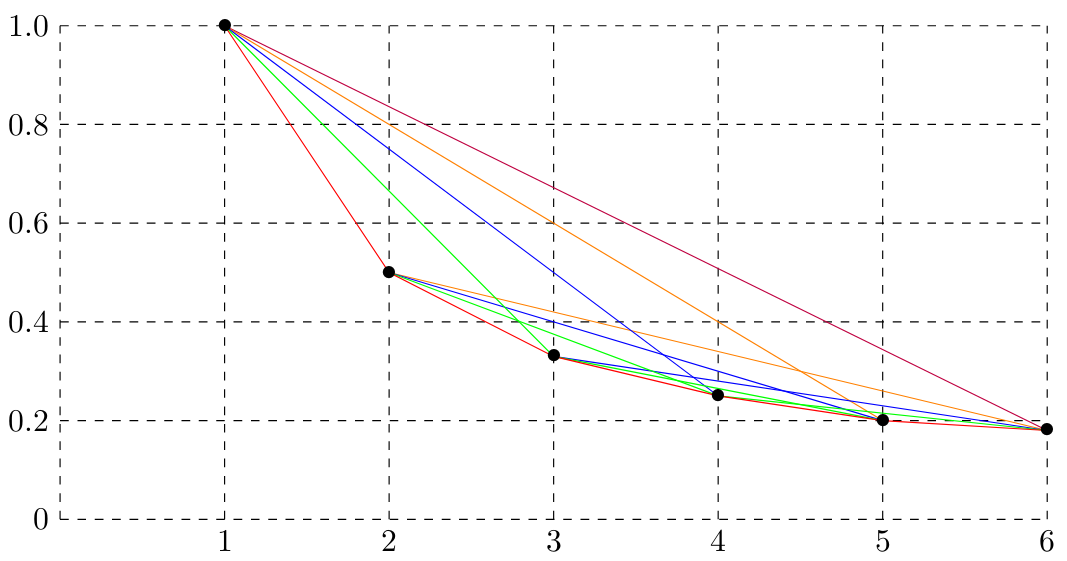}
    \caption{A set of $n = 6$ points which form a weak crescent configuration in $L^\infty$. Coordinates $\{ (1,1), (2,  1/2), (3,  1/3), (4, 1/4), (5, 1/5), (6, 1/6) \}$.}
    \label{fig:loobadlinelike}
\end{figure}

\end{example}

Example \ref{ex:loobadlinelike} demonstrates the usefulness of a stronger notion of general position. Note that a common feature of the three trivial configurations presented in Figures \ref{fig:l2_trivial} and \ref{fig:loobadlinelike} is that their distance graphs are isomorphic to the distance graph of equally spaced points on a line in the following sense.

\begin{defi}
Let $S, T \subset \R^2$ such that $|S| = |T| = n$ for some $n \in \N$. Let $||\cdot||_S, ||\cdot||_T$ be two norms in $\R^2$. We say that the \textbf{distance graphs of $\bm{S}$ in $\bm{||\cdot||_S}$ and $\bm{T}$ in $\bm{||\cdot||_T}$ are isomorphic} if there exists a bijection $\phi: S \to T$ such that for all $a,b,c,d \in S$ we have 
\[ ||a - b||_S = || c - d||_S \iff ||\phi(a) - \phi(b)||_T = ||\phi(c) - \phi(d)||_T.\]
\end{defi}

The choice of comparing an arbitrary distance graph to the distance graph of equally spaced points on a line is natural because equally spaced points on a line have the same structure in any normed space.

\begin{lemma}\label{lemma:equallyspaced}
Fix a norm $||\cdot||$. Let $S = \{ s_1, \dots, s_n \} \subset \R^2$ be a set of $n$ equally spaced points on a line for some $n \in \N$. In other words, $s_1, \dots, s_n$ lie on a common line and
\[ d_{||\cdot||}(s_1, s_2) = \dots = d_{||\cdot||}(s_i, s_{i+1}) = \dots = d_{||\cdot||}(s_{n-1}, s_n ).\]
Then for all $i,j \in \{ 1, \dots, n \}$ we have 
\[ d_{||\cdot||}(s_i, s_j) = |j - i|\cdot d_{||\cdot||}(s_1, s_2).\]
\end{lemma}
\begin{proof}
If $i = j$, clearly $d_{||\cdot ||}(s_i, s_j) = 0$. Without loss of generality, assume $i < j$. Because $s_1, \dots, s_n$ lie on a line, the vectors $s_2 - s_1$, $s_3 - s_2$, $\dots$, $s_n - s_{n-1}$ are linearly dependent. Moreover, $|| s_2 - s_1|| = ||s_3 - s_2|| = \dots = ||s_n - s_{n-1} ||$. By linear dependence,

\begin{align*}
    || s_j - s_i || &= || s_j - s_{j-1} || + ||s_{j-1} - s_{j-2} || + \dots + ||s_{i+1} - s_i || \\
    &= | j - i | \cdot || s_2 - s_1 ||.
\end{align*}
\end{proof}
Lemma \ref{lemma:equallyspaced} immediately implies the following. 

\begin{cor}\label{cor:equallyspaced}
Let $S,T \subset \R^2$ such that $|S| = |T| = n$ for some $n \in \N$. Suppose $S$ and $T$ are sets of equally spaced points on a line. Let $||\cdot||_S$, $||\cdot||_T$ be any two norms. Then the distance graphs of $S$ in $||\cdot||_S$ and $T$ in $||\cdot||_T$ are isomorphic. 
\end{cor}

Now we define line-like configurations. By Corollary \ref{cor:equallyspaced}, they are well-defined. 

\begin{defi}\label{def:general_line_like}
Fix a norm $||\cdot||$. A set of $n$ points in the plane is said to form a \textbf{line-like configuration in \bm{$||\cdot||$}} if its distance graph is isomorphic to the distance graph of $n$ equally spaced points on a line. 
\end{defi}

In the next section (Section \ref{section:crescentconfigurations}), we use the concept of line-like configurations to define strong general position and strong general configurations. The remaining content of this section consists of examples of line-like configurations. Line-like configurations are studied in depth in Section \ref{section3}. 

First, we provide simple examples of line-like configurations of size $n$, for every natural number $n$.

\begin{exa}\label{ex:linelike_two_three} $ $
\begin{enumerate} 
\item Trivially, in any norm, equally spaced points on a line form a line-like configuration. 
\item In $L^2$, equally spaced points on a circular arc form a line-like configuration. See Figure \ref{fig:l2_trivial}. 
\item In $L^\infty$, certain perturbations of equally spaced points on a line form a line-like configuration. See Example \ref{ex:loobadlinelike}.
\item In Theorem \ref{thm:linelinelike} and Theorem \ref{thm:circlelinelike} we provide constructions of line-like configurations in a broad class of norms. See Sections \ref{section:linelinelike} and \ref{section:circlelinelike}. 
\end{enumerate}
\end{exa}

Second, we describe all line-like configurations of size 2, 3 in an arbitrary norm.

\begin{exa} $ $
\begin{enumerate}
    \item Any two distinct points trivially form a line-like configuration in any norm. 
    \item Line-like configurations of size three correspond to (possibly degenerate) isosceles triangles. To construct a line-like configuration of size three in an arbitrary norm $||\cdot||$, start with distinct points $A,B \in \R^2$. Draw the $||\cdot||$ circle centered at $B$ with radius $|AB|$. Pick any point $C$ lying on this circle such that $|AC| \neq |AB|$. Then $|AB| = |BC|$ and $|AB| \neq |AC|$, so $A,B,C$ forms a line-like configuration. 
\end{enumerate}
\end{exa}

Finally, we classify line-like configurations of size 4 in strictly convex norms. 

\begin{lemma}\label{lemma:line_like_four}
Let $||\cdot||$ be a strictly convex norm. Let $A, B, C$ be a line-like configuration of size three. Then there exist exactly two points $D,E$ such that $ABCD$ and $ABCE$ are line-like configurations. Moreover, at least one of $ABCD$ and $ABCE$ is a parallelogram. 
\end{lemma}

\begin{proof}

Let $C_1$ be $B_{||\cdot||}(C, |BC| )$ and $C_2$ be $B_{||\cdot||}(B, |AB| )$. The set of points $X$ for which $ABCX$ is a line-like configuration is precisely the set of intersection points of the two circles $C_1$ and $C_2$. Translate $\overrightarrow{AC}$ to point $B$ and let the tip of the translated vector be $D$. Then $ABCD$ is a parallelogram with $|AB| = |BC| = |CD|$ and $|AC| = |BD|$. So $D$ lies on both circles $C_1$ and $C_2$. 

By Lemma \ref{lemma:twocircles}, $C_1$ and $C_2$ have at most two intersection points. We show that a second intersection point exists by a monotonicity argument. Let the intersection point of line $\overleftrightarrow{BC}$ with $C_1$ which is not $B$ be $R$. Let the intersection points of line $\overleftrightarrow{BC}$ with $C_2$ be $P$ and $Q$ so that $\overrightarrow{BQ}$ points in the same direction as $\overrightarrow{CB}$ and $\overrightarrow{BP}$ points in the same direction as $\overrightarrow{BC}$. Since $||\cdot||$ is strictly convex, $|AC| < 2 \cdot |AB|$, which implies $|CP| < |CR|$. On the other hand, because $\overrightarrow{BQ}$ points in the same direction as $\overrightarrow{CB}$, $|CQ| > |CB|$. Thus $C_1$ and $C_2$ intersect once in each upper half-plane above and below line $\overleftrightarrow{BC}$. These give the two intersection points $D$ and $E$. 
\begin{figure}[h!]
    \centering
    \includegraphics[scale=0.3]{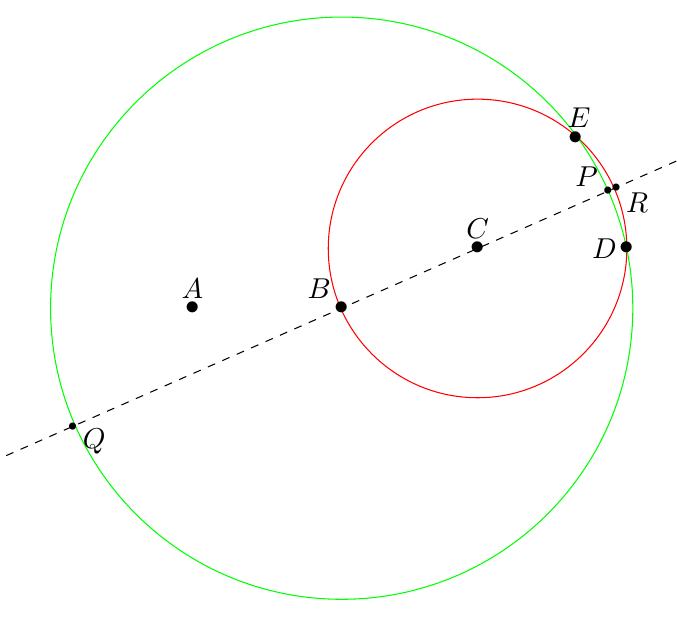}
    \includegraphics[scale=0.3]{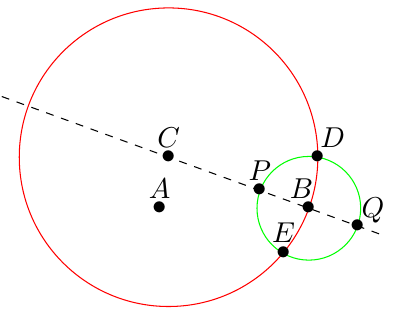}
    \caption{Illustrating the proof of Lemma \ref{lemma:line_like_four}. Here we use the $L^2$ norm. Given three points $A,B,C$ which form a line-like configuration, there exist exactly two points $D,E$ such that $ABCD$ and $ABCE$ form a line-like configuration.}
    \label{fig:linelike4_proof}
\end{figure}
\end{proof}


\subsection{Strong general position and strong crescent configurations}\label{section:crescentconfigurations} 

Using this notion of line-like configurations, we can define strong general position in an arbitrary norm $||\cdot||$. 

\begin{defi}\label{def:generalgeneralposition}
A set of points in the plane is said to lie in \textbf{strong general position in \bm{$||\cdot||$}} if the following three conditions hold.
\begin{enumerate}
    \item No three points lie on a common line.
    \item No four points lie on a common $||\cdot||$ circle.
    \item No four points form a line-like configuration of size four.
\end{enumerate}
\end{defi}

\begin{rem}\label{rem:l2stronggeneral}
The notion of $L^2$ strong general position, as given in Definition \ref{def:generalgeneralposition}, is more restrictive than the standard notion of $L^2$ general position, as given in Definition \ref{def:euclgeneralposition}. Specifically, strong general position additionally forbids line-like configurations of size four. Borrowing the notation from Lemma \ref{lemma:line_like_four}, let $ABCD$ be a line-like configuration which is not a parallelogram. By symmetry, the perpendicular bisectors of $AB$, $BC$, and $CD$ meet in a common point, so $A,B,C,D$ lie on a common circle. Thus, a set of points in $L^2$ general position lies in $L^2$ strong general position if and only if it does not contain a parallelogram $ABCD$ with $AB = BC = CD$, $AC = BD$, $AB || CD$ and $AC || BD$. 
\end{rem}

Finally, using this notion of general position in $||\cdot||$, we define strong crescent configurations in $||\cdot||$. 

\begin{defi}\label{def:generalcrescent} 
A set of $n$ points in the plane is said to form a \textbf{strong crescent configuration in $||\cdot||$} if the following three conditions hold.
\begin{enumerate}
    \item The $n$ points lie in strong general position in $||\cdot||$.
    \item The $n$ points determine $n - 1$ distinct distances.
    \item For each $1 \leq i \leq n - 1$, there exists a distance which occurs with multiplicity exactly $i$. 
\end{enumerate}
\end{defi}

Below, we collect examples of strong crescent configurations under various norms. 

\begin{exa}\label{ex:strong_construction}$ $
\begin{enumerate}
    \item In any norm, crescent configurations of size 2, 3 exist trivially. For constructions, see Example \ref{ex:linelike_two_three}. 
    \item For arbitrary $||\cdot||$, we construct a strong crescent configuration of size $4$. See Section \ref{section:crescentgeneral}. 
    \item Pal\' asti's \cite{Pal87, Pal89} constructions of crescent configurations of size $n \leq 5$ are strong. Additionally, Durst et al. \cite{DHHMP} construct many strong crescent configurations of size $n = 4,5$. However, known constructions of crescent configurations of size $6,7,8$ (due to Pal\' asti) are not strong. We construct a strong crescent configuration of size $6$. See Section \ref{section:crescentl2}.
    \item In $L^\infty$ (and thus in its dual norm $L^1$), we construct strong crescent configurations of size $4,5,6,7,8$. See Section \ref{section:crescentloo}.
\end{enumerate}
\end{exa}


\section{Constructions of line-like configurations}\label{section3}
In the previous section we defined line-like configurations (cf. Definition \ref{def:general_line_like}). Line-like configurations of size four are used when defining strong crescent configurations in $||\cdot||$. In this section, we provide constructions of line-like configurations of size $n$ for $n \geq 5$ in a broad class of norms.


\subsection{Line-like configurations in non-strictly convex norms}\label{section:linelinelike}

Recall that a norm is non-strictly convex if and only if its unit circle contains a line segment (Definition \ref{def:strictlyconvex}). In general, when studying distinct distances problems in normed spaces, it is not uncommon for non-strictly convex norms to have vastly different behavior compared to strictly convex norms. For example, consider the unit distances problem. Let $u_{||\cdot||}(n)$ denote maximum number of distances of length 1 that can be determined by $n$ points in $\R^2$ in the norm $|||\cdot||$. If $||\cdot||$ is strictly convex, then $u_{||\cdot||}(n) = O(n^{4/3} )$ \cite{Va}. If $||\cdot||$ is not strictly convex, then $u_{||\cdot||}(n) = \Theta(n^2)$ \cite{Br}. 

In the following result, for any non-strictly convex norm, we construct many line-like configurations which satisfy the property that no three points lie on a line. The key insight behind the proof is that in non-strictly convex norms, we can have $||x + y ||= ||x|| + ||y||$ without $x,y \in \R^2$ being linearly dependent. Thus there exist sets of points which have the additivity relations of equally spaced points on a line, even though the points do not lie on a common line. 

\linelinelike*

\begin{proof} See Figure \ref{fig:linelikeline}. Let $||\cdot||$ be a norm which is not strictly convex. Then its unit circle contains a line segment. Denote the (scaled and translated) copy of this line segment on a general circle $B_{||\cdot||}(p, r)$ by $\ell_{p, r}$. Pick a point $P_1$. For all $1 \leq i \leq n - 1$, pick a point $P_{i+1}$ lying on $\ell_{P_i, 1}$. Then for all $P_i, P_j \in \{ P_1, \dots, P_n \}$ we have $||P_j - P_i || = |j - i|$. Thus $\{ P_1, \dots, P_n \}$ is a line-like configuration. When picking each of the points $P_1, \dots, P_n$, there were infinitely many choices. Thus there are infinitely many such configurations. 

\begin{figure}[h!]
    \centering
    \includegraphics[scale=0.3]{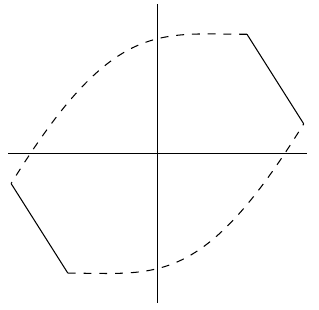}
    \includegraphics[scale=0.3]{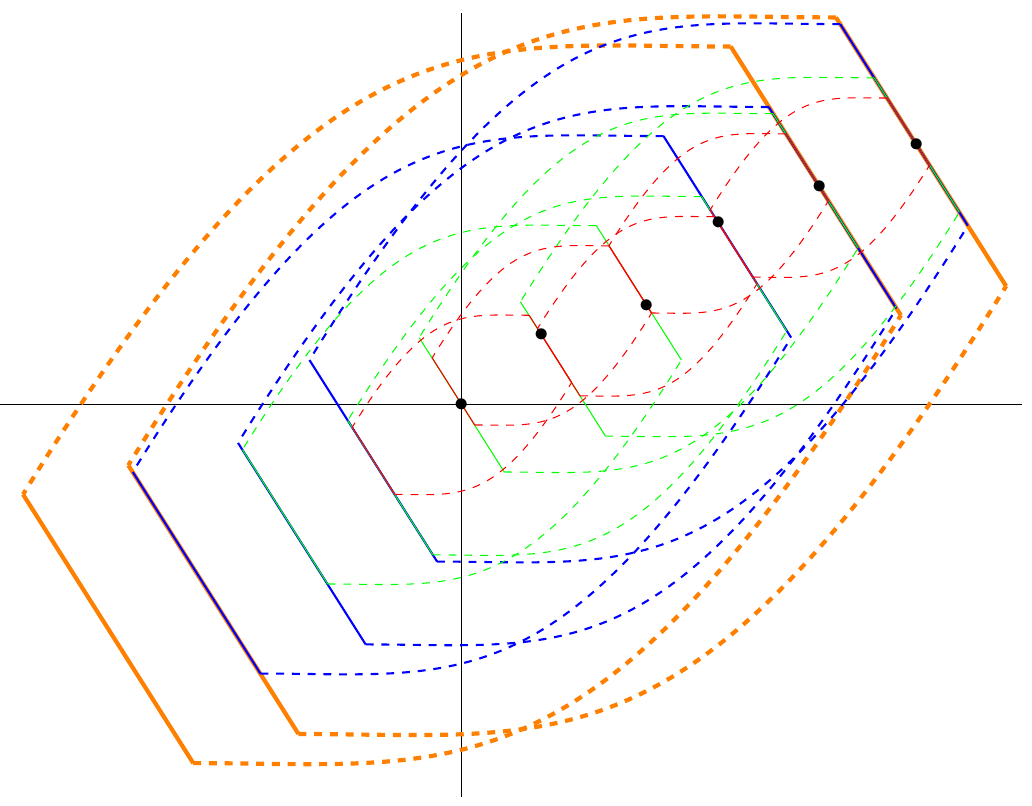}
    \caption{Left: a non-strictly convex norm. Right: Constructing a line-like configuration in a non-strictly convex norm. The red circles have radius 1, green circles have radius 2, the blue circles have radius 3, and the orange circles have radius 4.}
    \label{fig:linelikeline}
\end{figure}
\end{proof}

\begin{cor}\label{cor:linelinelike}
Let $||\cdot||$ be a norm which is not strictly convex. For each $n$, there exist infinitely many (after scaling and translating) line-like configurations of size $n$ in $||\cdot||$ which satisfy the property that no three points lie on a common line. 

\end{cor}

\begin{proof} Repeat the proof of Theorem \ref{thm:linelinelike}. When choosing the point $P_{i+1}$, pick any point lying on $\ell_{P_i, 1}$ as like before, but now exclude any point lying on a line determined by any two points in $\{ P_1, P_2, \dots, P_i \}$. Infinitely many such $P_{i+1}$ exist because there are only finitely many points on $\ell_{P_i, 1}$ which lie on a line with two points in $\{ P_1, P_2, \dots, P_i \}$, and there are infinitely many points on $\ell_{P_i, 1}$.
\end{proof}

\begin{example}
Because $L^1$ and $L^\infty$ are non-strictly convex, Theorem \ref{thm:linelinelike} and Corollary \ref{cor:linelinelike} apply. The class of examples produced by Theorem \ref{thm:linelinelike} for $L^\infty$ generalizes Example \ref{ex:loobadlinelike}.
\end{example}


\subsection{Line-like configurations in norms whose unit circles contain an $L^2$ origin arc}\label{section:circlelinelike}

In Section \ref{section:linelinelike}, we construct line-like configurations of any size in non strictly convex norms. These constructions rely on the fact that in a non strictly convex norm, two circles can intersect in infinitely many points. By contrast, in a strictly convex norm, two circles intersect in at most two points (Lemma \ref{lemma:twocircles}). By this heuristic, we expect line-like configurations in strictly convex norms to be rarer. 

For a particular class of strictly convex norms, we construct line-like configurations of any size. Specifically, we consider norms whose unit circles contain a \textit{$L^2$ origin arc}. 

\begin{defi}
Let $A$ be an arc of positive length on an $L^2$ circle centered at the origin. Then we say $A$ is a \textbf{\bm{$L^2$} origin arc}. 
\end{defi}

In $L^2$, equally spaced points along a circular arc form a line-like configuration (Figure \ref{fig:l2_trivial}). This construction can be generalized to norms whose unit circles contain an $L^2$ origin arc. \\

\noindent \textbf{Theorem \ref{thm:circlelinelike}.} \textit{Let $||\cdot||$ be a norm whose unit circle contains an $L^2$ origin arc. Then for each $n$, there exist infinitely many (after scaling and translating) line-like configurations of size $n$ in $||\cdot||$.}

\begin{proof}
See Figure \ref{fig:circlenorm} for a unit circle of a norm whose unit circle contains an $L^2$ origin arc. 

First we introduce some terminology. Let $O$ denote the origin. For $P \in \R^2$, let $c_{P,r}(\theta) := P + (r \cos(\theta), r \sin(\theta) )$ denote a parameterization of an $L^2$ circle centered at $P$ with radius $r$. For points $P,Q \in \R^2$ with $P \neq Q$, let $t(P,Q)$ denote the unique $\theta \in [0, 2\pi)$ for which $c_{P, |PQ|} = Q$. 

Let $n \in \N$. We are given that the unit circle of $||\cdot||$ contains an $L^2$ origin arc. Let this arc be parameterized by $c_{O, r}(\theta)$ for $\theta \in [\theta_1, \theta_2]$, with $0 \leq \theta_1 < \theta_2 \leq \pi$. Pick $0  < \epsilon \leq (\theta_2 - \theta_1)/(n-2)$. Set $P_1 = c_{O, 1}(\theta_1)$ and $P_i = c_{P_{i-1}, 1}(\theta_1 + (i - 2)\epsilon)$ for $i \in \{ 2, 3, \dots, n \}$. See Figures \ref{fig:l2linelike} and \ref{fig:l2arclinelike}. We claim that $P_1, \dots, P_n$ form a line-like configuration in $||\cdot||$, which satisfy the property that no three points lie on a common line. (Setting $\epsilon = 0$ gives $n$ equally spaced points on a line.) 

It suffices to show the following. 
\begin{enumerate}
    \item $P_1, \dots, P_n$ form a line-like configuration in $L^2$. 
    \item For all $1 \leq i < j \leq n$, we have $t(P_i, P_j) \in [\theta_1, \theta_2]$. 
\end{enumerate}

Proof of (1): For each $i \in \{ 1, 2, \dots, n - 3 \}$, note that $\angle P_i P_{i+1} P_{i+2} = \angle P_{i+1} P_{i+2} P_{i+3} = \pi - \epsilon$. Thus reflecting about the perpendicular bisector of $P_{i+1} P_{i+2}$ sends $P_i, P_{i+1}, P_{i+2}, P_{i+3}$ to $P_i, P_{i+1}, P_{i+2}, P_{i+3}$. So the perpendicular bisectors of $P_{i} P_{i+1}$, $P_{i+1} P_{i+2}$, and $P_{i+2} P_{i+3}$ intersect in a point, which means that $P_i, P_{i+1}, P_{i+2}, P_{i+3}$ lie on a common circle (cf. Remark \ref{rem:l2stronggeneral}). This implies that $P_1, \dots, P_n$ lie on a common $L^2$ circle. Since $|P_i P_{i+1} | = 1$ for each $i \in \{ 1, 2, \dots, n - 1 \}$, the points $P_1, \dots, P_n$ are equally spaced on their common $L^2$ circle.

Proof of (2): By rotational symmetry, $t(P_i, P_{i+k} ) = t(P_1, P_{k - 1} )$ for all $1 \leq k \leq n - 1$. Using (1) and angle chasing, it can be shown that $t(P_1, P_i) \leq t(P_1, P_{i+1})$ for all $2 \leq i \leq n - 1$. This implies $t(P_i, P_j) \leq t(P_{i+1}, j)$ and $t(P_i, P_j) \leq t(P_i, P_{j-1})$ for $i,j \in \{ 1, 2, \dots, n \}$ with $i + 1 < j$. Thus $\min_{i,j} t(P_i, P_j) = t(P_1, P_2) = \theta_1$ and $\max_{i,j} t(P_i, P_j) = t(P_{n-1}, n ) = \theta_2$. 

\end{proof}

\begin{figure}[h!]
    \centering
    \includegraphics[scale=0.3]{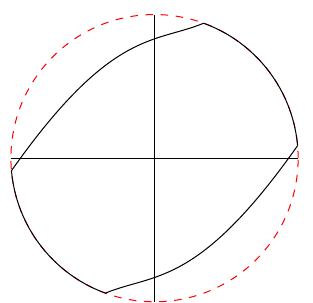}
    \caption{A norm whose unit circle intersects an $L^2$ origin arc.}
    \label{fig:circlenorm}
\end{figure}

\begin{figure}[h!]
    \centering
    \includegraphics[scale=0.22]{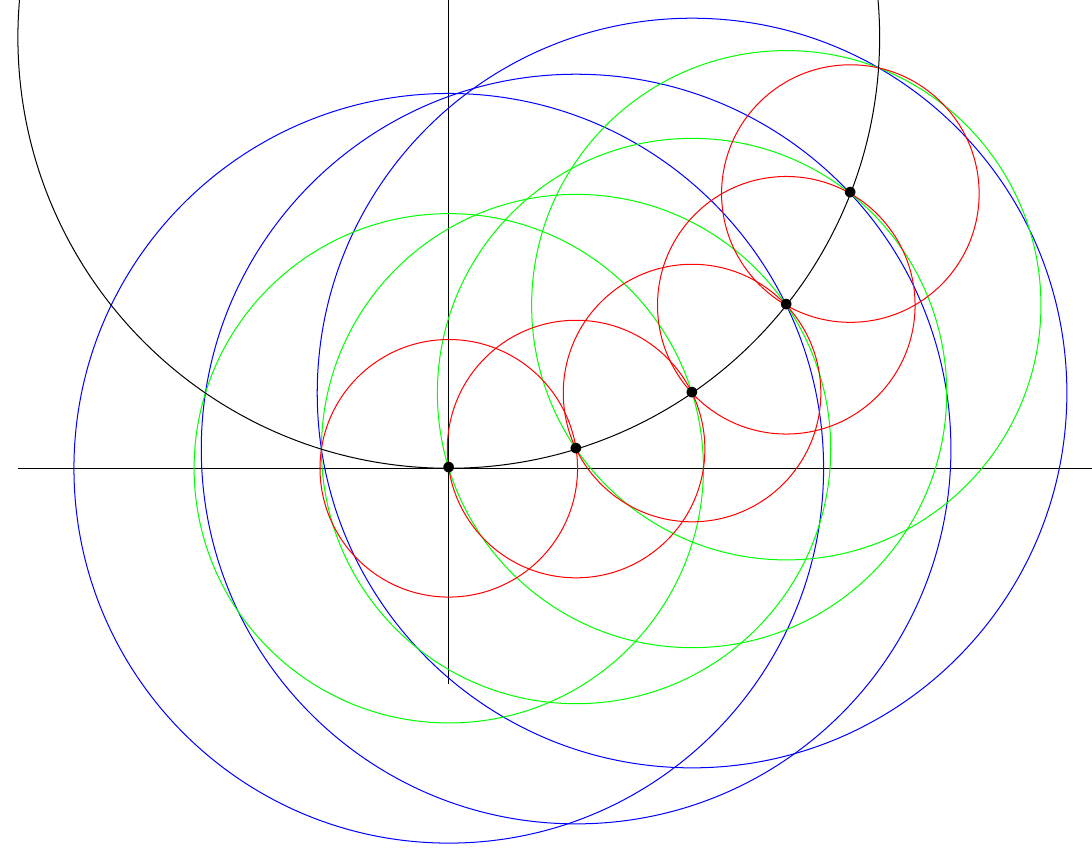}
    \caption{Constructing a line-like configuration in $L^2$.}
    \label{fig:l2linelike}
\end{figure}

\begin{figure}[h!]
    \centering
    \includegraphics[scale=0.22]{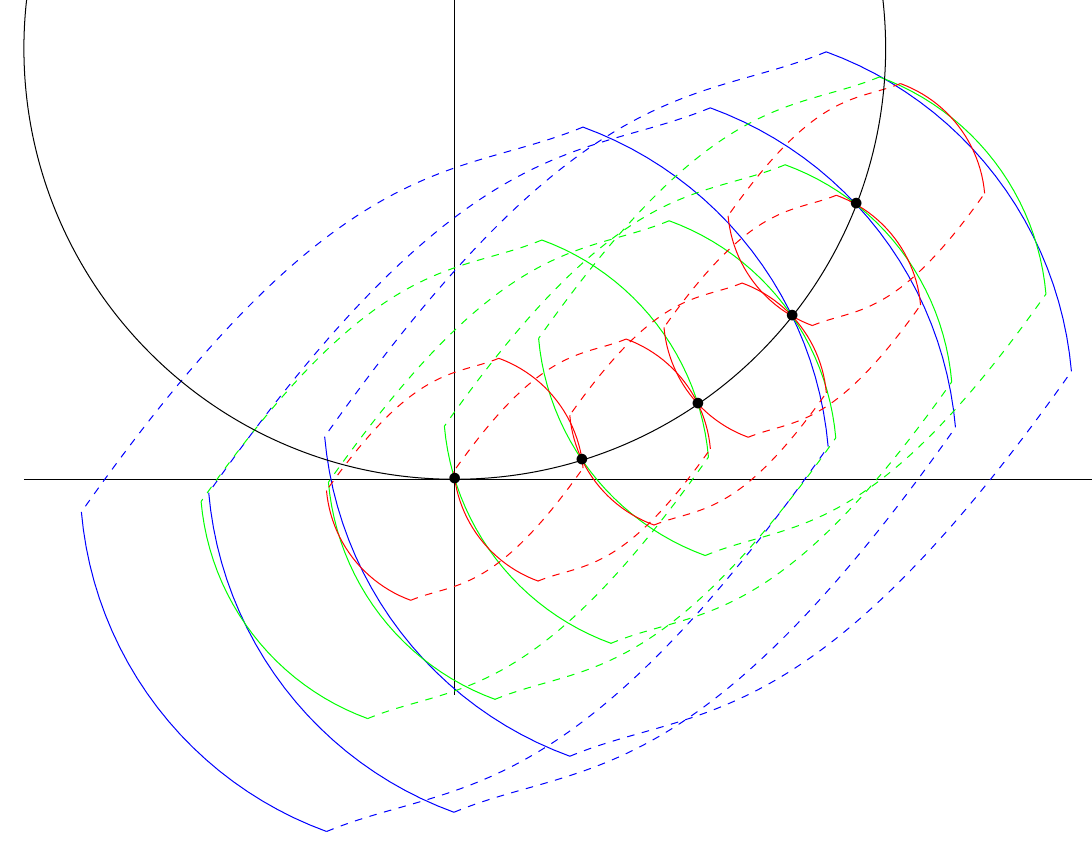}
    \caption{Constructing a line-like configuration in a norm whose unit circle contains an $L^2$ origin arc.}
    \label{fig:l2arclinelike}
\end{figure}

\begin{cor}\label{cor:circlelinelike}
Let $||\cdot||$ be a norm whose unit circle contains an $L^2$ origin arc. Then, for each $n$, there exist infinitely many (after scaling and translating) line-like configurations of size $n$ in $||\cdot||$, which satisfy the property that no three points lie on a common line. 
\end{cor}
\begin{proof}
Repeat the proof of Theorem \ref{thm:circlelinelike}. Because $P_1, \dots, P_n$ lie on a common $L^2$ circle, it follows that no three of $P_1, \dots, P_n$  lie on a common line. 
\end{proof}


\subsection{Line-like configurations in $L^p$, $1 < p < \infty$}\label{section:lplinelike}

We have numerically searched for line-like configurations in $L^p$. Of course as we will see in Theorem \ref{thm:crescentgeneral}, there are line-like configurations of size $n=4$. Whether these configurations may be extended to include a 5th point is a question of intersections of three $L^p$ balls, each ball given by one of the three previously specified distances. Numerically searching for such a configuration, we found no positive results for $p\neq2$, with arbitrarily small error as $p \to 2$. We employed two approaches in our search.

The first approach makes the ansatz that, if a line-like configuration were to exist in $L^p$, it would behave as $L^2$ and consist of $n$ equally spaced points along a unit ball. This provides a tremendous amount of structure to the potential collections of points, and we may naturally represent the location of $n$ points on the $L^p$ unit ball using $n$ angles $0 \leq t_i < 2\pi$ for $1 \leq i \leq n$ and the map $f:\mathbb{R} \to \mathbb{R}^2$ defined by $f(t)=\{\cos(t)^{\frac{2}{p}},\sin(t)^{\frac{2}{p}} \}$. We may arbitrarily renumber our $t_i$ so that they correspond to the ordering induced by a line-like configuration. Fixing a value on $t_1$ determines the location of the first point, and specifying $t_2$ provides the first order distance between the points $f(t_1)$ and $f(t_2)$. From these two values, $t_i$ for $i \geq 3$ are determined; $t_i$ corresponds to the unique point on the $L^p$ ball so that $d(f(t_i),f(t_{i-1}))=d(f(t_1),f(t_2))$ and $t_i \neq t_{i-2}$. This reduces our problem to a numerical search on two bounded variables, $t_1$ and $t_2$. Once the first order distance has determined the $t_i$, we may check the higher order distances to see if we have obtained a line-like configuration. This does not always produce a line-like configuration (cf. Example \ref{ex:lplinelike}). Numerically, it appears that this never produces a  line-like configuration, regardless of our choice of $t_1,t_2$, for $n\geq 5$ points, although the discrepancy in higher order distances goes to $0$ as $p \to 2$, as one might expect.

The second approach relaxes our ansatz but is computationally more intense. If we do not assume that the points lie on an $L^p$ ball, we may still specify the location of $n$ points using angles $t_1,\dots t_{n-1}\in[0,2\pi)$ and distance $d>0$. Then letting $x_0=\{0,0\}$, define $x_i=x_{i-1}+d \cdot f(t_i)$. These ensure that the first order distances are correct; then we may numerically compute the higher order distances and check for a crescent configuration. This algorithm must search over $n$ variables, and we were unsuccessful in finding configurations.

\begin{example}\label{ex:lplinelike}
Let $1 < p < \infty$. Consider the four points 
\[ x_1=(0,1),\ x_2= \left( \frac{1}{2^{1/p} }, \frac{1}{2^{1/p} } \right),\ x_3=(1,0),\ x_4= \left(\frac{1}{2^{1/p} }, -\frac{1}{2^{1/p} } \right). \]
These points lie on the $L^p$ circle of radius 1 centered at the origin. We compute 
\[ d_p (x_1,x_2)=d_p (x_2,x_3)=d_p (x_3,x_4).\] 
Also, $d(x_1,x_3)=2^{1/p}$ and $d(x_2,x_4)=(2^{p-1})^{1/p}=2^{1-1/p}$. Thus $d_p (x_1, x_3) = d_p(x_2, x_4)$ if and only if $p = 2$. 
\end{example}


\section{Classification of line-like crescent configurations in $L^\infty$}\label{section4}

In this section, we prove a structural result about line-like configurations in $L^\infty$. Specifically, we show that every line-like configuration of size $n \geq 7$ in $L^\infty$ satisfies at least one of the following three properties.
\begin{enumerate}
    \item Three points lie on a common line. 
    \item Four points lie on a common $L^\infty$ circle. 
    \item The set of $n$ points is a perpendicular perturbation of a horizontal or vertical line, i.e., has very similar structure to a set of $n$ equally spaced points on a horizontal or vertical line.
\end{enumerate}

This result is significant in that it is a ``crescent-type'' result. Rephrased, Erd\H{o}s' conjecture claims the following: There exists some $N$ for which, for all $n \geq N$, if a set of $n$ points satisfies the property that for each $1 \leq i \leq n - 1$ there exists a distance which occurs exactly $i$ times, then three points lie on a common line or four points lie on a common circle. We have proven the following: For all $n \geq 7$, if a set of $n$ points forms a line-like configuration in the $L^\infty$ norm, then three points lie on a common line, four points lie on a common $L^\infty$ circle, or the set of points is a perpendicular perturbation in $||\cdot||$. 


\subsection{Perpendicular perturbations and line-like crescent configurations}
In Section \ref{section:linelinelike}, we provide constructions of infinitely many line-like configurations of arbitrary size under any non-strictly convex norm $||\cdot||$. Note that each of these constructions has a simple structure---namely, it is a \textit{perpendicular perturbation} in $||\cdot||$. 

\begin{defi}\label{def:perp_perturb}
For each $n$, let $P_1, \dots, P_n$ and $Q_1, \dots, Q_n$ be points in the plane.
\begin{enumerate}
    \item We say that $P_1, \dots, P_n$ is a \textbf{perpendicular perturbation of $Q_1, \dots, Q_n$} if the lines $\overleftrightarrow{P_i Q_i}$ are parallel for all $1 \leq i \leq n$. (In other words, $P_1, \dots, P_n$ is a perpendicular perturbation of $Q_1, \dots, Q_n$ if there exists a line $\ell$ so that for all $1 \leq i \leq n$, $P_i$ and $Q_i$ are mapped to the same point when projected onto $\ell$.)
    \item We say that $P_1, \dots, P_n$ is a \textbf{perpendicular perturbation of $\ell$} if there exist equally spaced points $Q_1, \dots, Q_n$ on $\ell$ so that $P_1, \dots, P_n$ is a perpendicular perturbation of $Q_1, \dots, Q_n$. 
    \item Let $||\cdot||$ be a non-strictly convex norm, and for some $k$, let $\ell_1, \dots, \ell_k$ be lines which contain each of the line segments in the unit circle of $||\cdot||$. Let $\ell_i'$ be a line perpendicular to $\ell_i$ for each $1 \leq i \leq n$. We say that $P_1, \dots, P_n$ is a \textbf{perpendicular perturbation in $||\cdot||$} if $P_1, \dots, P_n$ is a perpendicular perturbation of $\ell_i'$ for some $1 \leq i \leq n$. 
\end{enumerate}
\end{defi}

\begin{exa}$ $
    \begin{enumerate}
        \item The set of points $\{ (1,1), (2,1/2), \dots, (n,1/n) \}$ (cf. Example \ref{ex:loobadlinelike}) is a perpendicular perturbation of the set of points $\{ (1,0), (2,0), \dots, (n,0) \}$, a perpendicular perturbation of the $x$-axis, and a perpendicular perturbation in $L^\infty$. 
        \item Let $||\cdot||$ be a non-strictly convex norm. Then every line-like configuration constructed by Theorem \ref{thm:linelinelike} is a perpendicular perturbation in $||\cdot||$.
    \end{enumerate}
\end{exa}

However, the following example shows that for all $n \geq 3$, there exist line-like configurations of size $n$ in $L^\infty$ which are not perpendicular perturbations in $L^\infty$.  

\begin{exa}\label{ex:non_perp_pert}
Fix $n \geq 3$. If $n = 2k + 1$ for $k \geq 1$, consider the set of $2k + 1$ points 
\[ \{ (0,0), (1,a), (1+a, 1+a), (2+a, 1+2a), (2+2a, 2+2a), \dots, ( k(1+a), k(1+a) ) \}\]
for some $0 < a < 1$. If $n = 2k$ for $k \geq 2$, consider the above set with the last point removed. The reader can check that this set of points forms a line-like configuration in $L^\infty$. However this set of points is not a perpendicular perturbation in $L^\infty$. To be a perpendicular perturbation in $L^\infty$, this set of points must be a perpendicular perturbation of a horizontal or vertical line. But the $x$-coordinates $0,1,1+a$ and the $y$-coordinates $0,a,1+a$ of the first three points respectively are not equally spaced, because $0 < a < 1$.
\end{exa}

Even though the set from Example \ref{ex:non_perp_pert} is not a perpendicular perturbation, its structure is similar to that of a perpendicular perturbation because it contains many points on a common line. Specifically, the points
\[ (0,0), (1+a, 1+a), \dots, ( \lfloor (n-1)/2 \rfloor (1+a), \lfloor (n-1)/2 \rfloor (1+a) ) \]
are equally spaced on a common line. 

When studying crescent configurations, we require that the points lie in some notion of general position in order to omit trivial configurations (cf. Section \ref{section:linelikedef}). Similarly, we omit trivial examples of line-like configurations by introducing \textit{line-like crescent configurations:}

\begin{defi}\label{def:line_like_crescent}
Fix a norm $||\cdot||$. A set of $n$ points is said to form a \textbf{line-like crescent configuration in $||\cdot||$} if the following three conditions hold.
\begin{enumerate}
    \item The $n$ points form a line-like configuration in $||\cdot||$.
    \item No three points lie on a common line.
    \item No four points lie on a common $||\cdot||$ circle. 
\end{enumerate}
\end{defi}

\begin{exa}\label{ex:crescent_line_like}
$ $
\begin{enumerate}
    \item For each $n$, $\{ (1,1), (2,1/2), \dots, (n,1/n)\}$ (cf. Example \ref{ex:loobadlinelike}) forms a line-like crescent configuration in $L^\infty$. 
    \item The set of points in Example \ref{ex:non_perp_pert} is a line-like crescent configuration of size $n$ if and only if $n \leq 4$. When $n \geq 5$, the set of points is not a line-like crescent configuration because the points $(0,0), (1+a, 1+a), (2+2a, 2+2a)$ lie on a common line.
\end{enumerate}
\end{exa}

We claim that there are only finitely many such exceptions in the following sense.

\crescentlinfty*

The following example shows that the $n \geq 7$ bound in the statement of Theorem \ref{thm:crescentlinfty} is tight. 

\begin{exa}\label{ex:counterexample_crescentlinfty}
For $3 \leq n \leq 6$, there exist line-like crescent configurations in $L^\infty$ which are not perpendicular perturbations in $L^\infty$, namely 
\begin{align*}
    &\{ (0,0), (1,a), (1+a, 1+a) \}\\
    &\{ (0,0), (1,a), (1+a, 1+a), (2+a, 1+2a) \}\\
    &\{ (0,0), (1,a), (1+b, 1+a), (2+b, 1+a+b), (2+a+b, 2+a+b) \} \\
    &\{ (0,0), (1,a), (1+b, 1+a), (2+b, 1+a+b), (2+2b, 2+a+b), (3+2b, 2+2a+b) \} 
\end{align*}
for $0 < a < b < 1$. For each (ordered) set of points, note that the differences between consecutive points alternate between $(1,c)$ and $(c,1)$, for $c \in \{ a, b \}$. Using notation from Section \ref{section:crescentlinfty_notation}, we say that these line-like configurations are type $xy$, $xyx$, $xyxy$, and $xyxyx$ respectively (cf. Definition \ref{def:type}). See also Lemma \ref{lemma:bad_xyxy}, which states that a line-like crescent configuration of size $n$ and type $xyxy\cdots$ must satisfy $n \leq 6$. 
\end{exa}

The proof of Theorem \ref{thm:crescentlinfty} is structured as follows. In Section \ref{section:crescentlinfty_notation}, we introduce notation used in the proof of Theorem \ref{thm:crescentlinfty}. In Section \ref{section:linfty_theorem:proof}, we state Lemmas \ref{lemma:bad_BB}, \ref{lemma:bad_xbx'y}, \ref{lemma:bad_xx'}, \ref{lemma:bad_xyxy'}, \ref{lemma:bad_xyxy}, \ref{lemma:bad_xyx'y'}, \ref{lemma:bad_BBB}, \ref{lemma:bad_xBx'_xBy}, and \ref{lemma:bad_xBbx'y}. Then we use these lemmas to prove Theorem \ref{thm:crescentlinfty}. In Section \ref{section:lemmas_proofs}, we prove the lemmas stated and used in Section \ref{section:linfty_theorem:proof}. 


\subsection{Types, realizability, $m$-extendability}\label{section:crescentlinfty_notation}
Throughout the rest of this section, we exclusively use the $L^\infty$ norm, and omit the specification ``in $L^\infty$'' when referring to distances, line-like configurations, and so on. Let $p = (p_x, p_y)$ and $q = (q_x, q_y)$ be points in $\R^2$. In this section, we denote their distance by $d(p,q) := d_{L^\infty}(p, q) = \max \{ |p_x - q_x|, |p_y - q_y| \}$. For points $p_1, \dots, p_n$, let $[p_1, \dots, p_n]$ denote the ordered list of these points. 

\begin{defi}
Let $[p_1, \dots, p_n]$ be a line-like configuration. For $1 \leq i \leq n - 1$, the \textbf{$\bm{i}$\tss{th} order distance} is given by $d(p_1, p_{1 + i} ) = d(p_2, p_{2 + i}) = \dots = d(p_{n - i}, p_{n} )$. When the line-like configuration $[p_1, \dots, p_n]$ is clear, we denote its $i$\tss{th} order distance by \textbf{$d_i$}. 
\end{defi}

Next we define the \textit{type} of a line-like configuration. 

\begin{defi}
Let $p = (p_x, p_y)$ and $q = (q_x, q_y)$ be distinct points. 
\begin{enumerate}
    \item If $|p_x - q_x| > |p_y - q_y|$ and $q_x > p_x$, we say that $[p,q]$ is \textbf{type $\bm{x}$}.
    \item If $|p_x - q_x| > |p_y - q_y|$ and $p_x > q_x$, we say that $[p,q]$ is \textbf{type $\bm{x'}$}.
    \item If $|p_y - q_y| > |p_x - q_x|$ and $q_y > p_y$, we say that $[p,q]$ is \textbf{type $\bm{y}$}.
    \item If $|p_y - q_y| > |p_x - q_x|$ and $p_y > q_y$, we say that $[p,q]$ is \textbf{type $\bm{y'}$}.
    \item If $|p_x - q_x| = |p_y - q_y|$, $q_x > p_x$ and $q_y > p_y$, we say that $[p,q]$ is \textbf{type $\bm{b_{xy}}$}.
    \item If $|p_x - q_x| = |p_y - q_y|$, $p_x > q_x$ and $q_y > p_y$, we say that $[p,q]$ is \textbf{type $\bm{b_{x'y}}$}.
    \item If $|p_x - q_x| = |p_y - q_y|$, $p_x > q_x$ and $p_y > q_y$, we say that $[p,q]$ is \textbf{type $\bm{b_{x'y'}}$}.
    \item If $|p_x - q_x| = |p_y - q_y|$, $q_x > p_x$ and $p_y > q_y$, we say that $[p,q]$ is \textbf{type $\bm{b_{xy'}}$}.
\end{enumerate}
We write $T := \{ x, x', y, y', b_{xy}, b_{x'y}, b_{x'y'}, b_{xy'} \}$.
\end{defi}

\begin{defi}\label{def:type}$ $
\begin{enumerate}
\item Let $[p_1, \dots, p_n]$ be a line-like configuration. The \textbf{type} of $[p_1, \dots, p_n]$ is a string $a_1 a_2 \dots a_{n - 1}$, with $a_i \in T$, where for all $1 \leq i \leq n - 1 $ we have that $[p_i, p_{i+1}]$ is type $a_i$. 
\item Let $a_1 a_2 \dots a_{n-1}$ and $c_1 c_2 \dots c_{k-1}$ be types. We say that the type $a_1 a_2 \dots a_{n-1}$ \textbf{contains} the type $c_1 c_2 \dots c_{k-1}$ if $c_1 c_2 \dots c_{k-1}$ is a substring of $a_1 a_2 \dots a_{n-1}$.  
\item Let $a_1 a_2 \dots a_{n-1}$ be a type. We say that the type $a_1 a_2 \dots a_{n-1}$ has \textbf{length} $n-1$. 
\end{enumerate}
\end{defi}

\begin{defi} Let $k \geq 2$. 
\begin{enumerate}
\item We say that the type $a_1 a_2 \dots a_{k-1}$ is \textbf{realizable} if there exists a line-like crescent configuration $[p_1, \dots, p_k]$ with type $a_1 a_2 \dots a_{k- 1 }$. 
\item Let $m \geq 1$. We say that $a_1 a_2 \dots a_{k-1}$ is \textbf{$m$-extendable} if there exist $a_k,a_{k+1}, \dots, a_{k + m - 1} \in T$ so that $a_1 a_2 \dots a_{k-1 } a_k a_{k+1} \dots a_{k + m - 1}$ is realizable. 
\end{enumerate}
\end{defi}

We conclude with a remark crucial to the logic of the proofs in Sections \ref{section:linfty_theorem:proof} and \ref{section:lemmas_proofs}.

\begin{rem}
Throughout the proofs of Theorem \ref{thm:crescentlinfty} and related lemmas, we frequently make use of the following symmetries of $L^\infty$. 
\begin{enumerate}
    \item The following are isometries of $L^\infty$: reflection about a horizontal line, reflection about a vertical line, reflection about a line with slope $\pm 1$. 
    \item A type $a_1 a_2 \dots a_{n-1}$ is realizable if and only if $a_{n-1} a_{n-2} \dots a_1$ is realizable. 
\end{enumerate}
The following are simple example arguments making use of these symmetries. 
\begin{itemize}
    \item Lemma \ref{lemma:bad_xx'} states that $xx'$ is not 2-extendable. By symmetry (1), Lemma \ref{lemma:bad_xx'} is equivalent to the statement that any one of $x'x$, $yy'$, and $y'y$ is not 2-extendable. 
    \item Symmetry (1) implies that type $xx'$ is symmetric about the $x$-axis in the following sense: there exists a natural bijection between sets of points realizing $xx'y$ and sets of points realizing $xx'y'$. Thus $xx'y$ is 1-extendable if and only if $xx'y'$ is. 
    \item Lemma \ref{lemma:bad_xx'} states that $xx'$ is not 2-extendable. By symmetry (2), $x'x$ is not 2-extendable. By symmetry (1), any type which contains $xx'$ as a substring must be of the form $axx'b$, where $a,b \in T \cup \{ \epsilon \}$. (Here, $\epsilon$ denotes the empty string.) 
\end{itemize}
In particular, when we write ``$a_1 a_2 \dots a_{n-1}$ (and reflections)'', we mean the collection of types equivalent to $a_1 a_2 \dots a_{n-1}$ under symmetries (1) and (2). For example, ``$xb_{xy}$ (and reflections)'' refers to the types $x b_{xy}$, $x b_{xy'}$, $x' b_{x'y}$, $x' b_{x'y'}$, $y b_{xy}$, $y b_{x'y}$, $y' b_{xy'}$, $y' b_{x'y'}$, $ b_{xy}x$, $ b_{xy'} x$, $ b_{x'y} x'$, $ b_{x'y'} x'$, $ b_{xy} y$, $b_{x'y} y$, $ b_{xy'} y'$, $ b_{x'y'} y'$.
\end{rem}


\subsection{Lemma statements and proof of Theorem \ref{thm:crescentlinfty}}\label{section:linfty_theorem:proof}

First we state the lemmas used in the proof of Theorem \ref{thm:crescentlinfty}. Their proofs are given in Section \ref{section:lemmas_proofs}. 

\begin{lemma}\label{lemma:bad_BB}
The types $b_{xy} b_{xy}$ and $b_{xy} b_{x'y'}$ (and reflections) are not realizable. 
\end{lemma}

\begin{lemma}\label{lemma:bad_xbx'y}
The type $x b_{x'y}$ (and reflections) is not 2-extendable. 
\end{lemma}

\begin{lemma}\label{lemma:bad_xx'}
There do not exist $s,t \in \{ x, x', y, y' \}$ so that $xx'st$ is realizable.
\end{lemma}

\begin{lemma}\label{lemma:bad_xyxy'}
$ $
\begin{enumerate}
    \item The type $xyxy'$ (and reflections) is not realizable.
    \item The type $xyx'y$ (and reflections) is not realizable.
\end{enumerate}
\end{lemma}

\begin{lemma}\label{lemma:bad_xyxy}
For some $n$, let $c_1 c_2 \dots c_{n-1}$ be a type with
\[ \begin{cases}
c_i = x & \text{if }i \equiv 1 \mod 2 \\
c_i = y & \text{if }i \equiv 0 \mod 2 \\
\end{cases}\]
If $c_1 c_2 \dots c_{n-1}$ (or reflections) is realizable, then $n \leq 6$. 
\end{lemma}

\begin{lemma}\label{lemma:bad_xyx'y'}
For some $n$, let $c_1 c_2 \dots c_{n-1}$ be a type with 
\[ \begin{cases}
c_i = x & \text{if }i \equiv 1 \mod 4 \\
c_i = y & \text{if }i \equiv 2 \mod 4 \\
c_i = x'& \text{if }i \equiv 3 \mod 4 \\
c_i = y'& \text{if }i \equiv 0 \mod 4
\end{cases}\]
If $c_1 c_2 \dots c_{n-1}$ (or reflections) is realizable, then $n \leq 5$. 
\end{lemma}

\begin{lemma}\label{lemma:bad_BBB}
For some $n$, let $c_1 c_2 \dots c_{n-1}$ be a type with $c_i \in \{ b_{xy}, b_{x'y}, b_{x'y'}, b_{xy'} \}$ for all $1 \leq i \leq n - 1$. Then $n \leq 4$, and the only possible values of $c_1 c_2 \dots c_{n-1}$ (up to reflection) are $b_{xy}$, $b_{xy} b_{x'y}$, and $b_{xy} b_{x'y} b_{xy} $.
\end{lemma}

\begin{lemma}\label{lemma:bad_xBx'_xBy}
Suppose $d_2 = 2$. For some $n$, let $c_1 c_2 \dots c_{n-1}$ be a type with $c_i \in \{ b_{xy}, b_{x'y}, b_{x'y'}, b_{xy'} \}$ for all $1 \leq i \leq n - 1$. 
\begin{enumerate}
    \item The type $x c_1 c_2 \dots c_{n-1} x'$ (and reflections) is not 1-extendable.
    \item The type $x c_1 c_2 \dots c_{n-1} y$ (and reflections) is not 1-extendable.
\end{enumerate}
\end{lemma}

\begin{lemma}\label{lemma:bad_xBbx'y}
Suppose $d_2 = 2$. For some $n$, let $c_1 c_2 \dots c_{n-1}$ be a type with $c_i \in \{ b_{xy}, b_{xy'} \}$ for all $1 \leq i \leq n - 1$. 
\begin{enumerate}
    \item There does not exist a $t \in \{ b_{xy}, b_{x'y}, b_{x'y'}, b_{xy'}, x \}$ so that $x c_1 c_2 \dots c_{n-1} b_{x' y} t$ is realizable. 
    \item There does not exist a $t \in \{ b_{xy}, b_{x'y}, b_{x'y'}, b_{xy'}, x \}$ so that $t x c_1 c_2 \dots c_{n-1} b_{x'y}$ is realizable. 
\end{enumerate}

\end{lemma}

Now we give the proof of Theorem \ref{thm:crescentlinfty}, which is restated below. 

\crescentlinfty*

\begin{proof}[Proof of Theorem \ref{thm:crescentlinfty}]
Let $[p_1,p_2, \dots, p_n]$ be a line-like crescent configuration of size $n \geq 7$ with type $A := a_1 a_2 \dots a_{n-1}$. Without loss of generality, scale $[p_1, p_2, \dots, p_n]$ so that the first order distance satisfies $d_1 = 1$. By the triangle inequality, the second order distance satisfies $d_2 \leq 2$. 

Suppose $d_2 < 2$. Suppose $A$ contains some $b_{xy}$ (or reflections). Without loss of generality $A$ must contain $t b_{xy}$ for some $t \in T$. The types $x b_{xy}$, $y b_{xy}$, $b_{xy'} b_{xy} $, $b_{x'y} b_{xy} $ give $d_2 = 2$, a contradiction. By Lemma \ref{lemma:bad_BB}, $b_{x'y'} b_{xy} $ and $b_{xy} b_{xy} $ are not realizable. Thus $A$ must contain $x' b_{xy}$ or $y' b_{xy}$. By Lemma \ref{lemma:bad_xbx'y}, $n \leq 5$. Contradiction. Thus $A$ only contains $\{ x,x',y,y' \}$. Because $d_2 < 2$, $A$ cannot contain $xx$ (and reflections). By Lemma \ref{lemma:bad_xx'}, $A$ cannot contain $xx'$ (and reflections). Thus $a_i \in \{ x, x' \}$ for even $i$ and $a_i \in \{ y, y' \}$ for odd $i$, or vice versa. By Lemma \ref{lemma:bad_xyxy'}, $A$ must be of the form $xyxy \dots$ or $xyx'y' \dots$. Finally, by Lemma \ref{lemma:bad_xyxy} and Lemma \ref{lemma:bad_xyx'y'}, we have $n \leq 6$ as desired. 

Thus $d_2 = 2$. Suppose $A$ contains at least two of $\{ x,x',y,y' \}$. Since $d_2 = 2$, $A$ cannot contain $xx'$ or $xy$ (and reflections). Thus $A$ contains $x c_1 c_2 \dots c_{k-1} x'$ or $x c_1 c_2 \dots c_{k-1} y$ (or reflections) for $k \geq 2$ and $c_i \in \{ b_{xy}, b_{x'y}, b_{x'y'}, b_{xy'} \}$ for all $1 \leq i \leq k - 1$. By Lemma \ref{lemma:bad_BBB}, $k \leq 4$, so by Lemma \ref{lemma:bad_xBx'_xBy}, $A$ must be of the form $x c_1 c_2 \dots c_{k-1} x'$ or $x c_1 c_2 \dots c_{k-1} y$ (or reflections). Thus $n \leq 6$. Otherwise, $A$ contains at most one of $\{ x, x', y, y' \}$, without loss of generality $x$. If $A$ does not contain $x$, then $n \leq 4$ by Lemma \ref{lemma:bad_BB}. Suppose $A$ contains $x$. Suppose $A$ contains $b_{x'y}$ or $b_{x'y'}$. Because $d_2 = 2$, $A$ cannot contain $x b_{x'y}$ or $x b_{x'y'}$. By Lemma \ref{lemma:bad_BBB} and Lemma \ref{lemma:bad_xBbx'y}, if $A$ contains $x c_1 c_2 \dots c_{n-1} b_{x' y}$ or $x c_1 c_2 \dots c_{n-1} b_{x' y'}$, then $n \leq 5$. Thus $A$ cannot contain $b_{x'y}$ or $b_{x'y'}$. Thus $A$ only contains $x$, $b_{xy}$, and $b_{xy'}$. This implies that $A$ is a perpendicular perturbation of a horizontal line. 
\end{proof}


\subsection{Proofs of Lemmas}\label{section:lemmas_proofs}

This section contains the proofs of Lemmas \ref{lemma:bad_BB}, \ref{lemma:bad_xbx'y}, \ref{lemma:bad_xx'}, \ref{lemma:bad_xyxy'}, \ref{lemma:bad_xyxy}, \ref{lemma:bad_xyx'y'}, \ref{lemma:bad_BBB}, \ref{lemma:bad_xBx'_xBy}, and \ref{lemma:bad_xBbx'y}. Their statements can be found in Section \ref{section:linfty_theorem:proof}. 

Next, we define notation used in the proofs of these lemmas. 

\begin{defi}
Let $a_1 a_2 \dots a_{n-1}$ be a realizable type. Suppose there are $k$ symbols $a_i$ for which $a_i \in \{ x, x', y, y' \}$ and $n-1-k$ symbols $a_i$ for which $a_i \in \{ b_{xy}, b_{x'y}, b_{x'y'}, b_{xy'} \}$. Let $i_1, i_2, \dots, i_k$ be the subsequence of indices $\{ 1, 2, \dots, n - 1 \}$ for which $a_{i_j} \in \{ x, x', y, y' \}$. Let $f_1, \dots, f_k \in \R$ so that $|f_i| < 1$ for each $1 \leq i \leq k$. 

We define a list $[p_1, p_2, \dots, p_n ]$ of $n$ points as follows. Set $p_1 = (0,0)$. For all $1 \leq i \leq n - 1$ we have the following.
\begin{itemize}
    \item If $a_i = x$, then $i = i_j$ for some $1 \leq j \leq k$. Set $p_{i+1} := p_i + (1, f_{j})$. 
    \item If $a_i = x'$, then $i = i_j$ for some $1 \leq j \leq k$. Set $p_{i+1} := p_i + (-1, f_{j} )$. 
    \item If $a_i = y$, then $i = i_j$ for some $1 \leq j \leq k$. Set $p_{i+1} := p_i + (f_{j}, 1)$. 
    \item If $a_i = y'$, then $i = i_j$ for some $1 \leq j \leq k$. Set $p_{i+1} := p_i + (f_{j}, -1)$. \item If $a_i = b_{xy}$, then set $p_{i+1} := p_i + (1, 1)$.
    \item If $a_i = b_{x'y}$, then set $p_{i+1} := p_i + (-1, 1)$.
    \item If $a_i = b_{x'y'}$, then set $p_{i+1} := p_i + (-1, -1)$.
    \item If $a_i = b_{xy'}$, then set $p_{i+1} := p_i + (1, -1)$.
\end{itemize}
We say that the type $a_1 a_2 \dots a_{n-1}$ \textbf{has coordinates $[p_1, p_2, \dots, p_n]_{f_1, \dots f_k}$}. The $f_1, f_2, \dots, f_k$ are called the \textbf{free variables} of $a_1 a_2 \dots a_{n-1}$. When the $f_1, \dots, f_k$ are clear, we write that $a_1 a_2 \dots a_{n-1}$ has coordinates $[p_1, p_2, \dots, p_n]$, or that $[p_1, p_2, \dots, p_n]$ are the coordinates of $a_1 a_2 \dots a_{n-1}$. 
\end{defi}

Let $s_1, s_2, \dots, s_n$ be a line-like crescent configuration with type $a_1 a_2 \dots a_{n-1}$. Say $a_1 a_2 \dots a_{n-1}$ has coordinates $[p_1, p_2, \dots, p_n]$. Then, up to translation, there exist free variables $f_1, \dots, f_k$ for which $[s_1, s_2, \dots, s_n] = [p_1, p_2, \dots, p_n]_{f_1, \dots, f_k}$.

\begin{exa}
The type $xyxy$ has coordinates
\[ [(0,0), (1,a), (1+b, 1+a), (2+b, 1+a+c), (2+b+d, 2+a+c)] \]
for free variables $|a|,|b|,|c|,|d| < 1$. If, for example, we show that there exist no $|a|,|b|,|c|,|d| <1$ so that 
\[ [(0,0), (1,a), (1+b, 1+a), (2+b, 1+a+c), (2+b+d, 2+a+c), (1+b+d, 3+a+c)],\]
then it follows that $xyxyb_{x'y}$ is not realizable. 
\end{exa}

\begin{defi}\label{defn:type_order_distance}
Let $a_1 a_2 \dots a_{n-1}$ be a realizable type with coordinates $[p_1, p_2, \dots, p_n]_{f_1, \dots, f_k}$. Because $a_1 a_2 \dots a_{n-1}$ is realizable, there exist $f_1, f_2, \dots, f_k$ for which $[p_1, p_2, \dots, p_n]_{f_1, \dots, f_k}$ is a line-like crescent configuration. For these choices of $f_1, \dots, f_k$, we notate the $i$\tss{th} order distance of $[p_1, p_2, \dots, p_n]_{f_1, \dots, f_k}$ as 
\[ D_{i,f_1, \dots, f_k}(a_1 a_2 \dots a_{n-1} ) := d_i = d(p_1, p_{1+i}) = \dots = d(p_{n-i}, p_n).\]
When the value of $D_{i,f_1, \dots, f_k}(a_1 a_2 \dots a_{n-1} )$ is independent of $f_1, \dots, f_k$, we write \textbf{$D_i(a_1 a_2 \dots a_{n-1} )$}. 
\end{defi}

\begin{rem}
We typically use Definition \ref{defn:type_order_distance} when the value of $D_i(a_1 a_2 \dots a_{n-1})$ is independent of the free variables $f_1, \dots, f_k$. For example, consider the type $xxy$. It has coordinates $[p_1, p_2, p_3, p_4] = [(0,0)$, $(1,a)$, $(2,a+b)$, $(2+c, 1+a+b)]$. We have $d(p_1, p_3) = \max \{ 2, |a+b| \}$ and $d(p_2, p_4) = \{ 1+c, 1+b \}$. Because $|a|,|b|,|c| < 1$, this implies $d(p_1, p_3) = 2$ and $d(p_2, p_4) < 2$. Thus, independent of free variables, we have $D_2(xx) = 2$ and $D_2(xy) < 2$. In particular, this shows that $D_2(xxy)$ is undefined. In other words, $xxy$ is not realizable.
\end{rem}

Finally, we prove Lemmas \ref{lemma:bad_BB}, \ref{lemma:bad_xbx'y}, \ref{lemma:bad_xx'}, \ref{lemma:bad_xyxy'}, \ref{lemma:bad_xyxy}, \ref{lemma:bad_xyx'y'}, \ref{lemma:bad_BBB}, \ref{lemma:bad_xBx'_xBy}, and \ref{lemma:bad_xBbx'y}. Their statements can be found in Section \ref{section:linfty_theorem:proof}. Throughout these proofs, we assume the first order distance $d_1 = 1$. Additionally, $a,b,c, d$ denote free variables of a type. In other words, $a,b,c,d \in \R$ with $|a|,|b|,|c|,|d| < 1$.

\begin{proof}[Proof of Lemma \ref{lemma:bad_BB}]
The type $b_{xy} b_{xy}$ has coordinates $(0,0), (1,1), (2,2)$. This is not realizable because these points lie on a common line.\\

The type $b_{xy} b_{x'y'}$ has coordinates $(0,0), (1,1), (0,0)$. This is not realizable because the points are not distinct. 
\end{proof}

\begin{proof}[Proof of Lemma \ref{lemma:bad_xbx'y}]
Since $D_2(x b_{x' y} ) < 2$, we have $d_2 < 2$. The types $xb_{x'y} t$ for $t \in \{ x', y, b_{xy}, b_{x'y'} \}$ are not realizable because $d_2 < 2$. The types $x b_{x' y} t$ for $t \in \{ b_{x'y}, b_{xy'} \}$ are not realizable by Lemma \ref{lemma:bad_BB}. It suffices to show that $xb_{x'y} x$ and $x b_{x' y} y'$ are not 1-extendable. \\

The type $x b_{x'y} x$ has coordinates $(0,0), (1,a), (0,1+a), (1, 1+a+b)$. Considering second order distances, we have $1+a = 1+b$, so $a = b$. Thus $x b_{x' y} x$ has coordinates $(0,0), (1,a), (0,1+a), (1,1+2a)$. Since $d_2 < 2$, the types $x b_{x'y} x t$ for $t \in \{ x, b_{xy}, b_{xy'} \}$ are not realizable. Additionally, $x b_{x' y} x t$ with $t \in \{ b_{x' y}, b_{x' y'} \}$ are not realizable, because the coordinates of these types have three points on a common line (three points with $x$-coordinate 0). Thus $x b_{x' y} xt$ is only realizable if $t \in \{ y, y' \}$. 
\begin{itemize}
    \item The coordinates of $xb_{x' y} xy$ are $(0,0), (1,a), (0,1+a), (1,1+2a), (1+b, 2+2a)$. Considering third order distances, we have $\max \{ 1, 1+a \} = \max \{ |b|, 2+a \}$. This is impossible because $2+a > 1$ and $2+a > 1+a$. Thus $xb_{x'y} xy$ is not realizable. 
    \item The coordinates of $x b_{x'y} xy'$ are $(0,0), (1,a), (0,1+a), (1,1+2a), (1+b, 2a)$. Considering third order distances, we have $\max \{ 1, 1+2a \} = \max \{ |b|, |a| \}$. This is impossible because $1 > |b|$ and $1 > |a|$. Thus $xb_{x'y} xy'$ is not realizable. 
\end{itemize}

The type $x b_{x' y} y'$ has coordinates $(0,0), (1,a), (0,1+a), (b,a)$. Considering second order distances, we have $1+a = 1 - b$, so $b = -a$. Thus $x b_{x' y} y'$ has coordinates $(0,0), (1,a), (0,1+a), (-a,a)$. If $a > 0$, these points lie on the circle with corners $(-a,0)$ and $(1,1+a)$. If $a < 0$, these points lie on the circle with corners $(0,a)$ and $(1,1+a)$. Contradiction. Thus $x b_{x'y} y'$ is not realizable. 
\end{proof}

\begin{proof}[Proof of Lemma \ref{lemma:bad_xx'}]
It suffices to show that $xx'xy$, $xx'y'y$, $xx'yx$, and $xx'yx$ are not realizable. \\

The type $xx'xy$ has coordinates $(0,0), (1,a), (0,a+b), (1,a+b+c), (1+d, 1+a+b+c)$. Considering third order distances, $\max \{ 1, |a+b+c| \} = \max \{ |d|, |1+b+c| \}$. Because $d_3 \neq d_1 = 1$, this implies $|a+b+c| = |1+b+c| > 1$. Suppose $1+b+c < 0$. Then $1+b+c < -1$, which implies $b+c < -2$, a contradiction. Thus $1+b+c > 1$. Suppose $a+b+c < 0$. Then $a+b+c < -1$, and since $1 + b + c > 1$, this implies $a < -1$, a contradiction. Thus $a+b+c > 1$. But then $a+b+c = 1+b+c$, which implies $a = 1$, a contradiction. Thus $xx'xy$ is not realizable. \\

The type $xx'y'y$ has coordinates $(0,0), (1,a), (0,a+b), (c, -1+a+b), (c+d, a+b)$. If $|a+b| \leq 1$, then four points lie on a circle. 
\begin{itemize}
    \item If $c > 0$, then $(0,0), (1,a), (0,a+b) , (c, -1+a+b)$ lie on a circle.
    \item If $c+d > 0$, then $(0,0), (1,a), (0, a+b), (c+d, a+b)$ lie on a circle. 
    \item Otherwise, $c \leq 0$ and $c + d < 0$. In this case, $(0,0), (0,a+b), (c, -1+a+b), (c+d, a+b)$ lie on a circle.
\end{itemize}
Thus $|a+b| > 1$. Considering second order distances, $|a+b| = \max\{ 1-c, 1-b \} = |c+d|$. If $a+b < 0$, then $-a-b \geq 1-b$, which implies $a \leq -1$, a contradiction. Similarly, if $c+d < 0$, then $-c-d \geq 1-c$, which implies $c \leq -1$, a contradiction. Thus $a+b > 0$ and $c+d > 0$. Since $|a+b| = |c+d| > 1$, this implies $a,b > 0$ and $c,d > 0$. But $\max \{ 1-c, 1-b \} > 1$ implies $b < 0$ or $c < 0$. Contradiction. Thus $xx'y'y$ is not realizable. \\

The type $xx'y'x$ has coordinates $(0,0), (1,a), (0,a+b), (c, -1+a+b), (1+c, -1+a+b+d)$. Considering second order distances, $|a+b| = \max \{ 1-c, 1+b \}= \max\{ 1+c, 1+d \}$. Since $d_2 \geq 1-c$, $d_2 \geq 1+c$, and $d_2 \neq d_1 = 1$, we have $d_2 > 1$. If $a+b > 0$, then $a+b \geq 1+b$, so $a \geq 1$, a contradiction. Thus $a+b < 0$, and since $|a+b| > 1$, we have $a+b < -1$. Considering third order distances, $\max \{ |c|, |-1 + a + b| \} = \max \{ |c|, |-1+b+d| \}$. Because $a+b < -1$, we have $|-1 + a + b| > 2$, so $d_3 = 1 - a- b = |-1+b+d|$. Since $a+b < -1$, we have $a,b < 0$. So $\max \{ 1-c, 1+b \} > 1$ implies $c < 0$, and $\max \{ 1+c, 1+d \} > 1$ therefore implies $d> 0$. Since $b < 0$ and $d > 0$, we have $-1 + b + d < 0$. Since $c < 0$ and $d> 0$, we have $|a+b| = 1+d$, which implies $2+d = 1 - b - d$. Rearranging gives $1 + 2d = -b$. Since $d > 0$, we have $1 + 2d > 1$, but $-b < 1$. Contradiction. Thus $xx'y'x$ is not realizable.\\ 

The type $xx'y'x'$ has coordinates $(0,0), (1,a), (0,a+b), (c,-1+a+b), (-1+c, -1+a+b+d)$. Considering second order distances, $|a+b| = \max \{ 1-c, 1-b \} = \max \{ 1-c, 1-d \}$. If $a+b < 0$, then $-a -b \geq 1-b$, which implies $a \leq -1$, a contradiction. Thus $a+b > 0$. Considering third order distances, $\max \{ |c|, |-1+a+b| \} = \max \{ 2-c, |-1 + b + d| \}$. Since $a+b > 0$ (and $a+b < 2$), we have $|-1 + a+b | < 1$. Since $|c| < 1$ and $|-1 + a + b| < 1$, we have $\max \{ |c|, |-1+a+b| \} < 1$. But $2 - c > 1$. Contradiction. Thus $xx'y'x'$ is not realizable. 
\end{proof}

\begin{proof}[Proof of Lemma \ref{lemma:bad_xyxy'}]$ $\\
\indent Proof of (1): The type $xyxy'$ has coordinates $(0,0)$, $(1,a)$, $(1+b, 1+a)$, $(2+b, 1+a+c)$, $(2+b+d, a+c)$. Considering third order distances, $\max \{ 2+b, |1+a+c| \} = \max\{ |1+b+d|, |c| \}$. This is impossible because $2+b > |c|$ and $2+b > |1+b+d|$. Thus $xyxy'$ is not realizable. \\

Proof of (2): The type $xyx'y$ has coordinates $(0,0)$ $(1,a)$ $(1+b, 1+a)$, $(b, 1+a+c)$, $(b+d, 2+a+c)$. Considering third order distances, $\max \{ |b|, |1+a+c| \} = \max \{ |1-b-d|, 2+c \}$. This is impossible because $2+c > |b|$ and $2+c > |1+a+c|$. Thus $xyxy'$ is not realizable.
\end{proof}

\begin{proof}[Proof of Lemma \ref{lemma:bad_xyxy}]
The type $xyxy$ has coordinates $(0,0)$, $(1,a)$, $(1+b, 1+a)$, $(2+b, 1+a+c)$, $(2+b+d, 2+a+c)$. Considering third order distances, $\max \{ 2+b, |1+a+c| \} = \max \{ |1+b+d|, 2+c \}$. Because $2+b > |1+b+d|$ and $2+c > |1+a+c|$, it follows that $2+b = 2+c$ is the third order distance. Thus $b = c$. 

Now consider the type $xyxyxy$. Note that it contains $xyxy$ (and its reflection $yxyx$) three times. Thus it has coordinates $(0,0)$, $(1,a)$, $(1+b, 1+a)$, $(2+b, 1+a+b)$, $(2+2b, 2+a+b)$, $(3+2b, 2+a+2b)$, $(3+2b+c, 3+a+2b)$. But $(1,a)$, $(2+b, 1+a+b)$, and $(3+2b, 2+a+2b)$ lie on a common line. Thus $xyxyxy$ is not realizable. So such a type $c_1 c_2 \dots c_{n-1}$ can only be realizable if $n \leq 6$. 
\end{proof}

\begin{proof}[Proof of Lemma \ref{lemma:bad_xyx'y'}]
The type $xyx'y'$ has coordinates $(0,0)$, $(1,a)$, $(1+b, 1+a)$, $(b, 1+a+c)$, $(b+d, a+c)$. Considering second order distances, 
\[ d_2 = \max\{ 1+b, 1+a \} = \max \{ 1-b, 1+c \} = \max \{ 1-d, 1-c \}.\]
Considering third order distances, 
\[ d_3 = \max\{ |b|, |1+a+c| \} = \max \{ |-1+b+d|, |c| \} \]
Because $d_2 \neq 1$, there are four cases. 
\begin{enumerate}
    \item Case $1+b = 1+c = 1-d > 1$, so $b = c = -d > 0$. Then $|-1+b+d| = 1$, so $d_3 = \max \{ |-1+b+d|, |c| \} = 1$, a contradiction. 
    \item Case $1+a = 1-b = 1-c > 1$, so $a = -b = -c > 0$. Then $|1+a+c| = 1$, so $d_3 = \max \{ |b|, |1+a+c| \} = 1$, a contradiction. 
    \item Case $1+a = 1-b = 1-d$, so $a = -b = -d > 0$. Then $|-1+b+d| > 1$, so $1+a+c = 1-b-d$, so $a+c = -b-d$, and since $a=-b=-d$, this implies $a = c = -b = -d > 0$. 
    \item Case $1 + a = 1 + c = 1-d$, so $a = c = -d > 0$. Then $|1+a+c| > 1$, so $1+a+c = 1-b-d$, so $a+c = -b-d$, and since $a=c=-d$, this implies $a = c = -b = -d > 0$.
\end{enumerate}
Thus in any possible case, $a = c = -b = -d$. This means that $xyx'y'$ has coordinates $(0,0), (1,a), (1-a, 1+a), (-a, 1+2a), (-2a, 2a)$. 

Now consider the type $xyx'y'x$. By the above (using on $xyx'y'$ and $yx'y'x$), $xyx'y'x$ has coordinates $(0,0), (1,a), (1-a, 1+a), (-a, 1+2a), (-2a, 2a), (1-2a, 3a)$. But the points $(1-a, 1+a), (-a, 1+2a), (1-2a, 3a)$ lie on a common line. Thus $xyx'y'x$ is not realizable. So such a type $c_1 c_2 \dots c_{n-1}$ can only be realizable if $n \leq 5$. 
\end{proof}

\begin{proof}[Proof of Lemma \ref{lemma:bad_BBB}]
Let $c_1 c_2 \dots c_{n-1} $ be a type with $c_i \in \{ b_{xy}, b_{x'y}, b_{x'y'}, b_{xy'} \}$. Without loss of generality, let $c_1 = b_{xy}$. By Lemma \ref{lemma:bad_BB}, $c_2 \in \{ b_{x'y}, b_{xy'} \}$, so without loss of generality $c_2 = b_{x'y}$. By Lemma \ref{lemma:bad_BB}, $c_3 \in \{ b_{xy}, b_{x'y'} \}$. The coordinates of $b_{xy} b_{x'y} b_{x'y'}$ are $(0,0)$, $(1,1)$, $(0,2)$, $(-1,1)$, which lie on a circle with corners $(-1,1)$ and $(1,1)$. Thus $c_3 = b_{xy}$. Finally, we claim that $c_1 c_2 c_3 = b_{xy} b_{x'y} b_{xy}$ is not 1-extendable. By Lemma \ref{lemma:bad_BB}, $c_4 \in \{ b_{x'y}, b_{xy'} \}$. 
\begin{itemize}
    \item The type $c_1 c_2 c_3 c_4 = b_{xy} b_{x'y} b_{xy} b_{x'y}$ has coordinates $(0,0)$, $(1,1)$, $(0,2)$, $(1,3)$, $(0,4)$. This is not realizable because $(0,0)$, $(0,2)$, $(0,4)$ lie on a common line.
    \item The type $c_1 c_2 c_3 c_4 = b_{xy} b_{x'y} b_{xy} b_{xy'}$ has coordinates $(0,0)$, $(1,1)$, $(0,2)$, $(1,3)$, $(2,2)$. This is not realizable because $(1,1)$, $(2,2)$, $(1,3)$, $(2,2)$ lie on a common circle with corners $(0,1)$ and $(2,3)$. 
\end{itemize}
\end{proof}

\begin{proof}[Proof of Lemma \ref{lemma:bad_xBx'_xBy}]
$ $\\
\indent Proof of (1): For some $n$, let $c_1 c_2 \dots c_{n-1}$ be a type with $c_i \in \{ b_{xy}, b_{x'y}, b_{x'y'}, b_{xy'} \}$ for all $1 \leq i \leq n - 1$. Consider $x c_1 c_2 \dots c_{n-1} x'$. Because $d_2 = 2$, we have $c_1 \in \{ b_{xy}, b_{xy'} \}$ and $c_{n-1} \in \{b_{x'y}, b_{x'y'} \}$. Thus $n \geq 3$. By Lemma \ref{lemma:bad_BBB}, $n \leq 4$. If $n = 4$, then by Lemma \ref{lemma:bad_BBB}, $c_1 = c_3 = c_{n-1}$. This is a contradiction, thus $n = 3$. 

When $n = 3$, we have $x c_1 c_2 \dots c_{n-1} x' = x c_1 c_2 x'$ for $c_1 \in \{ b_{xy}, b_{xy'} \}$ and $c_2 \in \{ b_{x'y}, b_{x'y'} \}$. By Lemma \ref{lemma:bad_BB}, $c_1 c_2 = b_{xy} b_{x'y}$ or $c_1 c_2 = b_{xy'} b_{x'y'}$. Without loss of generality (reflection about the $x$-axis), we can assume $c_1 c_2 = b_{xy} b_{x'y}$. Then $D_3(x b_{xy} b_{x'y} x') < 3$. 

We claim that $x b_{xy} b_{x'y} x'$ is not 1-extendable. Suppose $x b_{xy} b_{x'y} x' t$ is realizable for some $t \in T$. Because $d_2 = 2$, $t \in \{ x', b_{x'y}, b_{x'y'} \}$. But $D_3(b_{x'y} x' t) \geq 3$, which contradicts $d_3 < 3$. Thus $x b_{xy} b_{x'y} x'$ is not 1-extendable. \\

Proof of (2): For some $n$, let $c_1 c_2 \dots c_{n-1}$ be a type with $c_i \in \{ b_{xy}, b_{x'y}, b_{x'y'}, b_{xy'} \}$ for all $1 \leq i \leq n$. Consider $x c_1 c_2 \dots c_{n-1} y$. Because $d_2 = 2$, we have $c_1 \in \{ b_{xy}, b_{xy'} \}$ and $c_{n-1} \in \{ b_{xy}, b_{x'y} \}$. By Lemma \ref{lemma:bad_BBB}, $n \leq 4$. If $n = 2$, then $c_1 = b_{xy}$. If $n = 3$, then by Lemma \ref{lemma:bad_BB}, $c_1 c_2 = b_{xy} b_{x'y}$ or $c_1 c_2 = b_{xy'} b_{xy}$. If $n = 4$, then by Lemma \ref{lemma:bad_BBB}, $c_1 c_2 c_3 = b_{xy} b_{x'y} b_{xy}$ or $c_1 c_2 c_3 = b_{xy} b_{xy'} b_{xy}$. We treat each of these cases individually to show that $x c_1 c_2 \dots c_{n-1} y$ is not 1-extendable. 
\begin{itemize}
    \item Case $x c_1 c_2 \dots c_{n-1} y = x b_{xy} y$. Suppose $x b_{xy} y t$ is realizable for some $t \in T$. Since $d_2 = 2$, $t \in \{ y, b_{xy}, b_{x'y} \}$. Then $D_3(x b_{xy} y) < 3$ but $D_3(b_{xy} y t) = 3$. Contradiction, thus $x b_{xy} y$ is not 1-extendable. 
    \item Case $x c_1 c_2 \dots c_{n-1} y = x b_{xy} b_{x'y} y$. Then $D_3(x b_{xy} b_{x'y} ) < 3$ but $D_3(b_{xy} b_{x'y} y ) = 3$. Contradiction, thus $x b_{xy} b_{x'y} y$ is not even realizable, and in particular not 1-extendable. 
    \item Case $x c_1 c_2 \dots c_{n-1} y = x b_{xy'} b_{xy} y$. Then $D_3(x b_{xy'} b_{xy} ) = 3$ but $D_3(b_{xy'} b_{xy} y ) < 3$. Contradiction, thus $x b_{xy'} b_{xy} y$ is not even realizable, and in particular not 1-extendable. 
    \item Case $x c_1 c_2 \dots c_{n-1} y = x b_{xy} b_{x'y} b_{xy} y$. Suppose $x b_{xy} b_{x'y} b_{xy} y t$ is realizable for some $t \in T$. Since $d_2 = 2$, $t \in \{ y, b_{xy}, b_{x'y} \}$. Then $D_3(x b_{xy} b_{x'y} ) < 3$, but $D_3(b_{xy} y t) = 3$. Contradiction, thus $x b_{xy} b_{x'y} b_{xy} y$ is not 1-extendable. 
    \item Case $x c_1 c_2 \dots c_{n-1} y = x b_{xy} b_{xy'} b_{xy} y$. Then $D_4(x b_{xy} b_{xy'} b_{xy} ) = 4$ but $D_4(b_{xy} b_{xy'} b_{xy} y ) < 4$. Contradiction, thus $x b_{xy} b_{xy'} b_{xy} y$ is not even realizable, and in particular not 1-extendable. 
\end{itemize}
\end{proof}

\begin{proof}[Proof of Lemma \ref{lemma:bad_xBbx'y}]
For some $n$, let $c_1 c_2 \dots c_{n-1}$ be a type with $c_i \in \{ b_{xy}, b_{xy'} \}$ for all $1 \leq i \leq n - 1$. If $x c_1 c_2 \dots c_{n-1} b_{x'y}$ is realizable, by Lemma \ref{lemma:bad_BBB}, we have $n \leq 3$. If $n = 3$, then $c_1 = b_{x'y}$. If $n = 2$, then $x c_1 c_2 \cdots c_{n-1} b_{x'y} = x c_1 b_{x'y}$, and by Lemma \ref{lemma:bad_BB}, $c_1 = \{ b_{xy}, b_{x'y'} \}$. We show that two of these three cases give types which are not realizable. 
\begin{itemize}
    \item Case $n = 3$ and $c_1 = b_{x'y}$. The type is $x b_{x' y} c_2 b_{x' y}$. We are given $d_2 = 2$, but $D_2(x b_{x'y} ) < 2$. Contradiction. Thus $x b_{x'y} c_2 b_{x'y}$ is not realizable.
    \item Case $n = 2$ and $c_1 = b_{x'y'}$. The type is $x b_{x'y'} b_{x'y}$. We are given $d_2 = 2$, but $D_3(x b_{x'y'} ) < 2$. Contradiction. Thus $x b_{x'y'} b_{x'y}$ is not realizable.
\end{itemize}
Thus $x c_1 c_2 \dots c_{n-1} b_{x'y}$ is realizable only if $n = 2$ and $c_1 = b_{xy}$. Now we prove (1) and (2). \\

Proof of (1): Suppose there exists a $t \in \{ b_{xy}, b_{x'y}, b_{x'y'}, b_{xy'} , x \}$ for which $x b_{xy} b_{x'y} t$ is realizable. By Lemma \ref{lemma:bad_BBB}, $t = b_{xy}$. Then $D_3(x b_{xy} b_{x'y} ) < 3$ but $D_3(b_{xy} b_{x'y} b_{xy} ) = 3$. Contradiction. Thus $x b_{xy} b_{x'y} t$ is not realizable. \\

Proof of (2): Suppose there exists a $t \in \{ b_{xy}, b_{x'y}, b_{x'y'}, b_{xy'} , x \}$ for which $t x b_{xy} b_{x'y}$ is realizable. Since $d_2 = 2$, $t \in \{ x, b_{xy}, b_{xy'} \}$. Then $D_3(x b_{xy} b_{x'y} ) < 3$, but $D_3(t x b_{xy} ) = 3$. Contradiction. Thus $t x b_{xy} b_{x'y}$ is not realizable.
\end{proof}


\section{Constructions of strong crescent configurations}\label{section5}

In this section, we provide constructions of strong crescent configurations. For an overview of known constructions, see Example \ref{ex:strong_construction}. In Section \ref{section:crescentgeneral}, we construct a strong crescent configuration of size 4 in any norm $||\cdot||$. In Section \ref{section:crescentl2}, we construct a strong crescent configuration of size $6$ in $L^2$. In Section \ref{section:crescentloo}, we construct strong crescent configurations of sizes $n \leq 8$ in $L^1$ and $L^\infty$. 


\subsection{Strong crescent configuration of size $n = 4$ in any norm}\label{section:crescentgeneral}

\crescentgeneral* 

\begin{proof} There are two cases, depending on whether the unit circle of $||\cdot||$ is a union of line segments. 

Case 1: The unit circle of $||\cdot||$ is not a union of line segments. Then pick a point $D$. Draw a unit circle centered at $D$, and pick points $A$, $B$, $C$ on this unit circle so that $A,B,C$ do not lie on a common line, $|AB| = |BC|$, and $|AD| > |AC| > |AB|$. See Figure \ref{fig:roundcrescent}.

Case 2: The unit circle of $||\cdot||$ is a union of line segments. Then pick a point $D$. Draw a unit circle centered at $D$. This circle contains at least one corner, i.e. a point where two line segments of different slopes meet. Let $B$ be a corner point. Let $A$ and $C$ each lie on one of the two line segments which meet at $B$, such that $|AB| = |BC|$ and $|AD| > |AC| > |AB|$. See Figure \ref{fig:cornercrescent}. 

In either case, we have the following: the distance $|AD| = |BD| = |CD|$ occurs three times, $|AB| = |BC|$ occurs two times, and $|AC|$ occurs once. Moreover, no three points of $A,B,C,D$ lie on a line, and the points $A,B,C,D$ do not form a line-like configuration. Thus $A,B,C,D$ form a strong crescent configuration.

\begin{figure}[h!]
    \centering
    \includegraphics[scale=0.3]{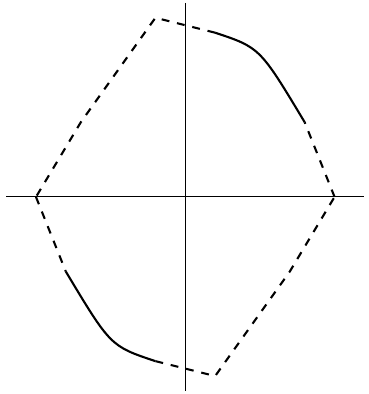}
    \includegraphics[scale=0.3]{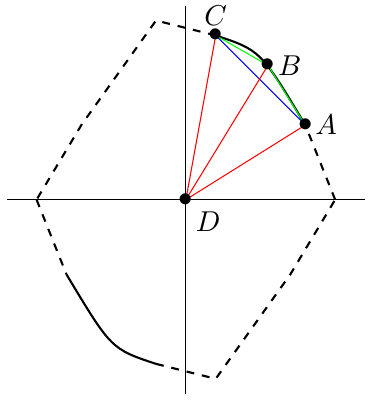}
    \caption{Left: a unit ball which is not a union of line segments. Right: a strong crescent configuration of size four in this norm.}
    \label{fig:roundcrescent}
\end{figure}

\begin{figure}[h!]
    \centering
    \includegraphics[scale=0.3]{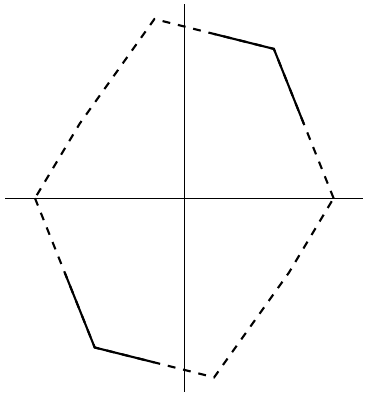}
    \includegraphics[scale=0.3]{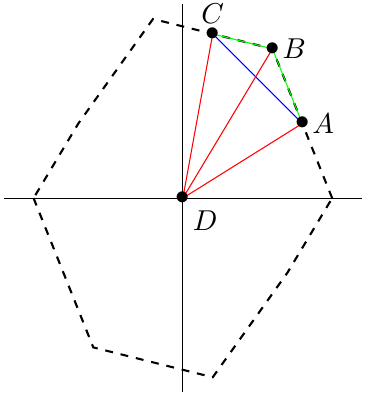}
    \caption{Left: a unit ball which is a union of line segments. Right: a strong crescent configuration of size four in this norm.}
    \label{fig:cornercrescent}
\end{figure}
\end{proof}


\subsection{Strong crescent configurations in $L^2$}\label{section:crescentl2}

For $n \leq 5$, there exist known constructions of crescent configurations in $L^2$ which are also strong crescent configurations. However,  Pal\' asti's \cite{Pal87, Pal89} constructions of crescent configurations of sizes 6, 7, 8 are not strong. We construct a strong crescent configuration of size 6 and conjecture that strong crescent configurations of sizes exceeding 6 do not exist. 

\ltwocrescent*

\begin{proof}
For constructions of sizes $n \leq 5$, see \cite{Liu, Pal87, Pal89, Erd89}. 

We constructed a strong crescent configuration of size 6 by searching a triangular lattice was using a backtracking algorithm.\footnote{
Our code can be found at \url{https://github.com/the-set-of-sets/nin}.
} The following are coordinates of a strong configuration of size 6 produced by our code: 
\[ \left\{ \left(\frac{1}{2}, \frac{\sqrt{3}}{2} \right), \left(1, \sqrt{3} \right), \left(\frac{3}{2}, \frac{\sqrt{3}}{2} \right), \left(\frac{7}{2}, \frac{2\sqrt{3}}{2} \right), \left(\frac{9}{2}, \frac{3\sqrt{3}}{2} \right), \left(4, 0 \right) \right\}.
 \]
The distance graph is depicted in Figure \ref{fig:l2crescent6}. 

\begin{figure}[h!]
    \centering
    \includegraphics[scale=0.4]{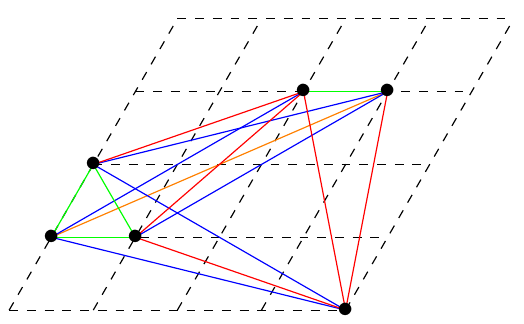}
    \caption{Strong crescent configuration of size 6 in $L^2$.}
    \label{fig:l2crescent6}
\end{figure}
\end{proof}

\begin{rem}\label{rem:l2search}
We exhaustively searched a $10 \times 10$ triangular lattice and showed that it does not contain a strong crescent configuration of size $7$ or $8$ in $L^2$. The search was conducted using a desktop computer with an Intel i7-6700K processor and 16GB RAM, and the duration was roughly 400 hours. 
\end{rem}


\subsection{Strong crescent configurations in $L^\infty$}\label{section:crescentloo}

\linftycrescent*

\begin{proof}
A square lattice was searched using a backtracking algorithm.\footnote{
Our code can be found at \url{https://github.com/the-set-of-sets/l1_linfty}.
} In the following table, we list a strong crescent configuration of size $n$ produced by our code for $4 \leq n \leq 8$:

\begin{center}
    \begin{tabular}{c|l}
         $n $ & Strong crescent configuration of size $n$, in $L^\infty$  \\
         \hline 
         4 & $\{ (0, 0), (0, 1), (1, 1), (1, 3) \} $ \\
         5 & $\{ (0, 0), (0, 1), (1, 1), (1, 3), (2, 4) \} $ \\
         6 & $\{ (0, 0), (0, 1), (1, 3), (2, 1), (2, 4), (4, 5) \} $ \\
         7 & $\{ (0, 0), (0, 4), (1, 2), (2, 3), (3, 1), (5, 4), (6, 6) \} $ \\
         8 & $\{ (0, 0), (0, 6), (1, 3), (2, 4), (3, 2), (4, 1), (5, 5), (6, 7) \} $
    \end{tabular}
\end{center}

The distance graphs are depicted in Figures \ref{fig:loocrescent456} and \ref{fig:loocrescent78}. 

\begin{figure}[h!]
    \centering
    \includegraphics[scale=0.3]{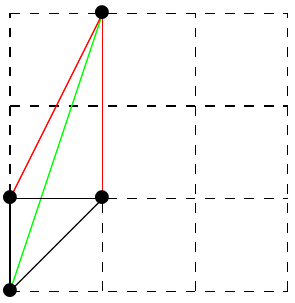}
    \includegraphics[scale=0.3]{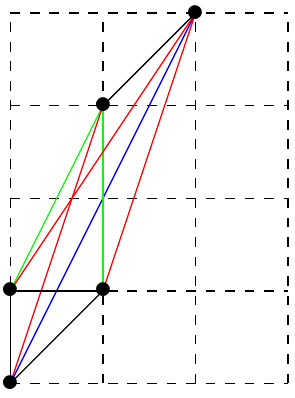}
    \includegraphics[scale=0.3]{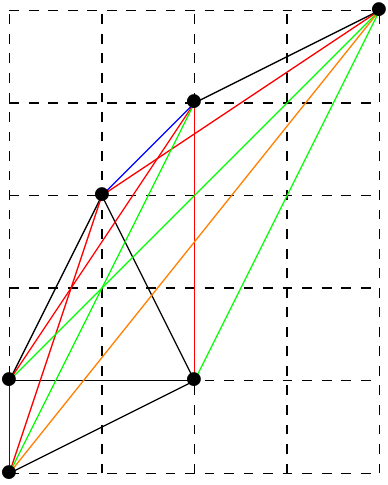}
    \caption{Strong crescent configurations of size 4, 5, 6 in $L^\infty$.}
    \label{fig:loocrescent456}
\end{figure}

\begin{figure}[h!]
    \centering
    \includegraphics[scale=0.3]{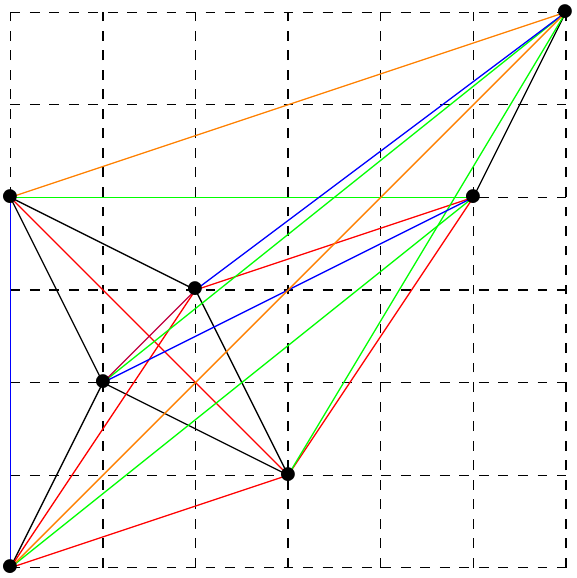}
    \includegraphics[scale=0.3]{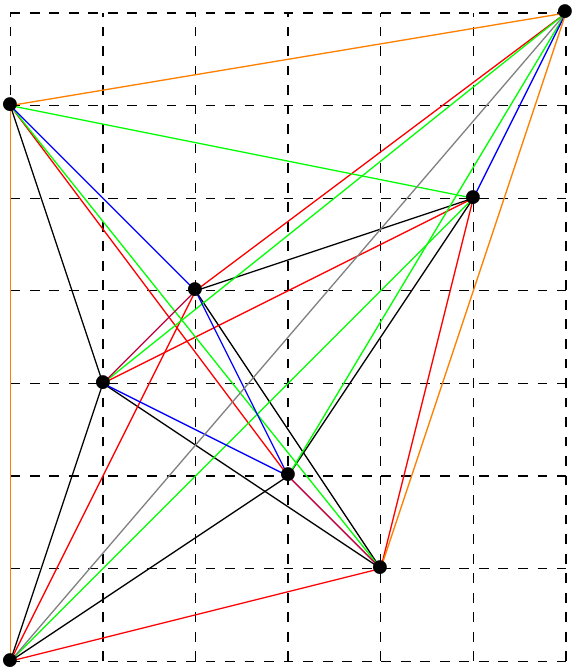}
    \caption{Strong crescent configurations of size 7, 8 in $L^\infty$.}
    \label{fig:loocrescent78}
\end{figure}

\end{proof}

\begin{rem}\label{rem:linftysearch}
We exhaustively searched a $9 \times 9$ lattice and showed that it does not contain a strong crescent configuration of size 9 in $L^\infty$. The search was conducted using a desktop computer with an Intel i3-7100 processor and 16GB RAM, and the duration was roughly 11 hours. 
\end{rem}

By Lemma \ref{lemma:duality}, $L^1$ and $L^\infty$ are dual norms in $\R^2$. This gives a correspondence between strong crescent configurations in $L^1$ and $L^\infty$. Let $P$ be a set of points in $L^\infty$. Let $P'$ be the set consisting of every point in $P$ rotated by $45^\circ$ counterclockwise. Then $P$ is a strong crescent configuration if and only if $P'$ is a strong crescent configuration. \\

\noindent \textbf{Corollary \ref{lonecrescent}. }\textit{In the $L^1$ norm, there exist strong crescent configurations of sizes $n \leq 8$.}\\


\section{Future work}\label{section6}


\subsection{Disproving the existence of large (strong) crescent configurations}
The main open question about crescent configurations is the following: for which $n$ do there exist crescent configurations of size $n$? The only known constructions of crescent configurations are of sizes $n \leq 8$ \cite{Liu, Pal87, Pal89, Erd89}. Attempts to find larger crescent configurations via computer search have been unsuccessful \cite{BGMMPS}. Erd\H{o}s \cite{Erd89} conjectured the following. 

\begin{conj}[Erd\H{o}s, 1989]\label{con:erdoscrescent}
For sufficiently large $N$, there do not exist crescent configurations of size $n$. 
\end{conj}

The same question can also be posed about strong crescent configurations. Given a norm $||\cdot||$, for which $n$ do there exist strong crescent configurations of size $n$ in $||\cdot||$? In Section \ref{section5}, we provide explicit constructions of strong crescent configurations of sizes $n \leq 8$ in $L^1$ and $L^\infty$, and of sizes $n \leq 6$ in $L^2$. Moreover, we performed computer searches and showed that certain lattice regions do not contain larger strong crescent configurations in $L^\infty$ and $L^2$ (cf. Remarks \ref{rem:l2search}, \ref{rem:l1search}, \ref{rem:linftysearch}). Extending Conjecture \ref{con:erdoscrescent}, we conjecture the following. 

\begin{conj}\label{con:strongcrescent}
Fix a norm $||\cdot||$. For sufficiently large $N$, there do not exist strong crescent configurations of size $n$ in $||\cdot||$. 
\end{conj}

We also pose a strengthening of Conjecture \ref{con:strongcrescent} for strictly convex norms. By Corollary \ref{lemma:line_like_four}, given three points $A,B,C$ which form a line-like configuration of size three in $||\cdot||$, there exist exactly two points $D,E$ so that $ABCD$ and $ABCE$ are line-like configurations in $||\cdot||$. In particular, at least one of $ABCD$ and $ABCE$ is a parallelogram. We can modify condition (3) of the definition of strong general position (Definition \ref{def:generalgeneralposition}) so that instead of forbidding all line-like configurations of size four, we forbid line-like configurations of size four which are not non-rectangular parallelograms. Call the sets of points which satisfy the corresponding Definition \ref{def:generalcrescent} \textit{weak crescent configurations}. 

\begin{conj}\label{con:weakcrescent}
In any strictly convex norm $||\cdot||$, we conjecture that for sufficiently large $N$, there do not exist weak crescent configurations of size $n$ in $||\cdot||$.
\end{conj}

In $L^2$, the weak crescent configurations are precisely the crescent configurations (cf. Remark \ref{rem:l2stronggeneral}).  When $||\cdot||$ is the $L^2$ norm, Conjecture \ref{con:weakcrescent} is equivalent to Conjecture \ref{con:erdoscrescent}. 

\subsection{Disproving the existence of large line-like configurations in most norms}\label{section:uglylinelike}

By Theorem \ref{thm:linelinelike} and Theorem \ref{thm:circlelinelike}, if the unit circle of $||\cdot||$ contains a line segment or an arc contained in an $L^2$ circle centered at the origin, then there exist infinitely many (after scaling and translating) line-like configurations of size $n$ in $||\cdot||$. These constructions are generalizations of the two line-like configurations in $L^2$: equally spaced points on a line and on a circle (cf. Figure \ref{fig:l2_trivial}). An interesting question is whether there exist arbitrarily large line-like configurations which do not rely on the structure of a straight line or $L^2$ arc. We conjecture that this is not the case. 

\begin{conj}\label{con:uglylinelike}
Let $||\cdot||$ be a norm whose unit circle does not contain a line segment or an arc contained in an $L^2$ circle centered at the origin. Then for sufficiently large $N$, the only line-like configurations of size $n \geq N$ in $||\cdot||$ are $n$ equally spaced points on a line.
\end{conj}

In Section \ref{section:lplinelike}, we provide numerical evidence toward this conjecture for the special case of the $L^p$ norm. 


\subsection{Classifying line-like crescent configurations in non-strictly convex norms}

In Theorem \ref{thm:crescentlinfty}, we prove that every line-like crescent configuration of size $n \geq 7$ in the $L^\infty$ norm must be a perpendicular perturbation in $L^\infty$. We are particularly interested in generalizing our result to other norms. 

\begin{conj}\label{con:crescent_nonstrictly}
Let $||\cdot||$ be a norm which is non strictly convex. Then there exists some $N$ for which every line-like crescent configuration of size $n \geq N$ in $||\cdot||$ is a perpendicular perturbation in $||\cdot||$. 
\end{conj}

For the definition of line-like crescent configurations, see Definition \ref{def:line_like_crescent}. For the definition of perpendicular perturbations, see Definition \ref{def:perp_perturb}.


\subsection{Extensions to higher dimensions}
In this paper, we only consider normed spaces $(\R^2, ||\cdot||)$. Burt et al. \cite{BGMMPS} considered a generalization of the problem of Euclidean crescent configurations to higher dimensions. They provided constructions of crescent configurations of size $n$ in $\R^{n - 2}$ for all $n \geq 3$. Can the notion of higher dimensional crescent configurations be appropriately generalized to arbitrary norms $||\cdot||: \R^n \to \R$?


\vspace{0.3cm}
\end{document}